\documentclass[11pt]{article}

\usepackage[english]{babel}
\usepackage[utf8x]{inputenc}
\usepackage[colorinlistoftodos]{todonotes}
\usepackage{float}

\usepackage{amsmath, amssymb, amsthm, cleveref}
\usepackage{graphicx}

\usepackage{mathrsfs}

\usepackage{multicol}
\usepackage{footmisc}

\renewcommand{\baselinestretch}{1.25}

 \newtheorem{theorem}{Theorem}[section]
 
 \newtheorem{conjecture}[theorem]{Conjecture}
 \newtheorem{lemma}[theorem]{Lemma}
 \newtheorem{proposition}[theorem]{Proposition}

 \theoremstyle{definition}
 \newtheorem{definition}[theorem]{Definition}
 \newtheorem{example}[theorem]{Example}
 \theoremstyle{remark}
 \newtheorem{remark}[theorem]{Remark}

\numberwithin{equation}{section}

  \newcommand{\subjclass}[1]{\footnotesize{{\em 2020 AMS Mathematics
 Subject Classification:} {\bf #1}}}
 \newcommand{\keywords}[1]{\renewcommand{\baselinestretch}{1}
 \footnotesize{\noindent {\em Key words and phrases:} {#1}}}

 \newcommand{\Nat}{\mathbb{N}}

 \newcommand{\set}[1]{\left\{#1\right\}}
 \newcommand{\Set}[2]{\set{#1\ \vert\ #2}}

\newcommand{\Mod}[3]{#1\equiv #2 \pmod{#3}}
\newcommand{\nMod}[3]{#1\not\equiv #2\ (\textrm{mod}\ #3)}
\newcommand{\Div}[2]{#1\,\mid\, #2}
\newcommand{\nDiv}[2]{#1\,\nmid\, #2}

\definecolor{darkmagenta}{rgb}{0.8,0,0.8}
\definecolor{darkgreen}{rgb}{0.2,0.7,0.02}
\definecolor{darkbrown}{rgb}{0.8,0.45,0}
\definecolor{bblue}{rgb}{0.25,0.25,0.75}

\date{}

\title{Prime Multiple Missing Graphs}

\author{Shamik Ghosh\thanks{An abridged version of this article appeared in Prime multiple missing graphs, Arai, K.(eds) Intelligent Computing. CompCom 2025. Lecture Notes in Networks and Systems, vol 1423. Springer, Cham. https://doi.org/10.1007/978-3-031-92602-0\_11}\\
{\footnotesize Department of Mathematics, Jadavpur University, Kolkata 700032, India. ghoshshamik@yahoo.com}
}
\begin{document}
\maketitle              

\begin{abstract}
The famous Goldbach conjecture remains open for nearly three centuries. Recently Goldbach graphs are introduced to relate the problem with the literature of Graph Theory. It is shown that the connectedness of the graphs is equivalent to the affirmative answer of the conjecture. Some modified version of the graphs, say, near Goldbach graphs are shown to be Hamiltonian for small number of vertices. In this context, we introduce a class of graphs, namely, prime multiple missing graphs such that near Goldbach graphs are finite intersections of these graphs. We study these graphs for primes $3,5$ and in general for any odd prime $p$. We prove that these graphs are connected with diameter at most $3$ and Hamiltonian for even $(>2)$ vertices.  Next the intersection of prime multiple missing graphs for primes $3$ and $5$ are studied. We prove that these graphs are connected with diameter at most $4$ and they are also Hamiltonian for even $(>2)$ vertices. We observe that the diameters of finite Goldbach graphs and near Goldbach graphs are bounded by $5$ (up to $10000$ vertices). We believe further study on these graphs with big data analysis will help to understand structures of near Goldbach graphs.\\[1em]
\\
\end{abstract}



\section{Introduction}

The famous Goldbach conjecture states that ``any even positive integer greater than $5$ is a sum of two odd primes.'' The problem remains open since 1742. Recently the conjecture is connected with Graph Theory \cite{DGGS,HL}. The Goldbach graph is defined in 2021 and it is shown that the above conjecture is true if and only if all finite Goldbach graphs are connected \cite{DGGS}. The idea is generated from a graph, called odd-even graph. Let $\mathcal{E}$ be the set of all non-negative even integers and $\mathcal{O}$ be the set of all positive odd integers. Let $A \subseteq \mathcal{E}$ and $O \subseteq \mathcal{O}$. An {\em odd-even graph} $\mathcal{G}_A(O)$ is a simple undirected graph with vertex set $A$ and two vertices $a,b\in A$ are adjacent if and only if $\frac{a+b}{2}, \frac{|a-b|}{2} \in O$. The (infinite) Goldbach graph is an odd-even graph where $A=\mathcal{E}$ and $\mathcal{O}$ is the set of all odd primes. For the finite Goldbach graphs $G_n$, $A=\set{0,2,4,\ldots,2n}$ for some natural number $n$ and $O$ is the set of all odd primes less than $2n$. The odd-even graph with $A=\set{2,4,\ldots,2n}$ and $O$ as the set of odd primes (less than $2n$) along with $1$ are observed \cite{DGGS} to be Hamiltonian for small even $n>2$ and a Hamiltonian path is exhibited for $A=\set{2,4,\ldots,1000}$. We call these graphs as `{\em near Goldbach graphs}'. 
 
We note that prime numbers are not `non-trivial' (more than $1$) multiples of its predecessors. So if we delete multiples of primes one by one, we will be nearer to the study of near Goldbach graphs. These ideas and facts motivate us to define a class of graphs, namely, `prime multiple missing graphs' which is an odd-even graph where the even set is the set of even positive integers and the odd set contains all odd integers except `non-trivial' multiples of a fixed prime number. Naturally, finite intersections of these graphs (restricted to finite vertex sets) are near Goldbach graphs. Thus the study of connectedness and Hamiltonicity of these graphs will play crucial role which may lead to better understanding of Goldbach conjecture in future. Also a surprising fact, we observe that the diameters of finite Goldbach graphs and near Goldbach graphs are bounded by $5$ (up to $10000$ vertices). Thus finding the diameter also becomes important.
 
In this paper, we study prime multiple missing graphs for primes $3$ and $5$. We prove both of them are connected as well as Hamiltonian for even ($>2$) number of vertices. They have Hamiltonian paths for odd number of vertices. We obtain structure theorems, describe the pattern of adjacency matrices and cycle structures. We show that the prime multiple missing graph is of diameter at most $3$ for primes $3$ and $5$. Next we consider intersections of prime multiple missing graphs for primes $3$ and $5$. The structure of these graphs are more complicated. However, we could be able to prove that these graphs are connected with diameter at most $4$ and also have the Hamiltonian property as before. In conclusion, we briefly describe the general structure of the prime multiple missing graphs for any prime $p$. The study of intersections for larger number of prime multiple missing graphs needs big data analysis. We are not sure how far the Hamiltonian property would carry over but as per our observation we propose a conjecture at the end regarding diameters of Goldbach graphs and near Goldbach graphs. 
 
Throughout the paper let $\Nat$ denote the set of all natural numbers. For verification and other studies on Goldbach's conjecture one may consult \cite{chen,RICH,TAO,VIN}. For graph theoretic concepts, definitions and terminologies, see \cite{DBW}.

\section{Prime Multiple Missing Graphs}

Let $n\in\Nat$. For convenience, we call a multiple $nk$ $(k\in\Nat)$, a {\em non-trivial multiple} of $n$ if $k>1$. We recall that $\mathcal{E}$ is the set of all non-negative even integers and $\mathcal{O}$ is the set of all positive odd integers.

\begin{definition}\label{defgpn}
Let $p$ be a prime number. Let us define an odd-even graph $G(p,n)=(V,E)$ with the even set $A=\Set{x\in\mathcal{E}}{2\leq x\leq 2n}$ and the odd set \\
$O=\Set{y\in\mathcal{O}}{1\leq y\leq 2n-1 \text{ and } y\neq pk \text{ for any } k>1,\, k\in\Nat}$ for some $n\in\Nat$.\\
Thus the vertex set $V=\set{2,4,6,\ldots,2n}$ and two vertices $a,b\in V$ are adjacent if and only if  $\frac{a+b}{2}$ and $\frac{|a-b|}{2}$ both are odd positive integers but not a non-trivial multiple of $p$. Then $G(p,n)$ is called a {\em prime multiple missing graph}. 
\end{definition}

\begin{remark}\label{rem1}
It is important to note that, since the graphs $G(p,n)=(V,E)$ are odd-even graphs, they are bipartite with partite sets\\
 $X=\Set{x\in V}{\Mod{x}{0}{4}}$ and $Y=\Set{x\in V}{\Mod{x}{2}{4}}$.
\end{remark}

\begin{example}\label{ex1:318}
Let us consider the prime multiple missing graph $G(3,18)$. Here $A=\set{2,4,6,\ldots,36}$ and $O=\set{1,3,5,7,11,13,17,19,23,25,29,31,35}$. The adjacency between vertices, i.e., the edges of the graph $G(3,18)$ are better described in Figure \ref{gd318} (right). There is a path\\
$(32,26,20,14,8,2,4,10,16,22,28,34)$ in the graph and the set $\set{12,24,36}\cup \set{6,18,30}$ is an independent set. Here bold lines indicate that the vertex in the path is adjacent to all vertices within rounds, e.g., the vertex $8$ is adjacent to $6,18,30$ and the vertex $10$ is adjacent to $12,24,36$. 
\end{example}

\begin{figure}[h]
\begin{center}
\includegraphics[scale=0.5]{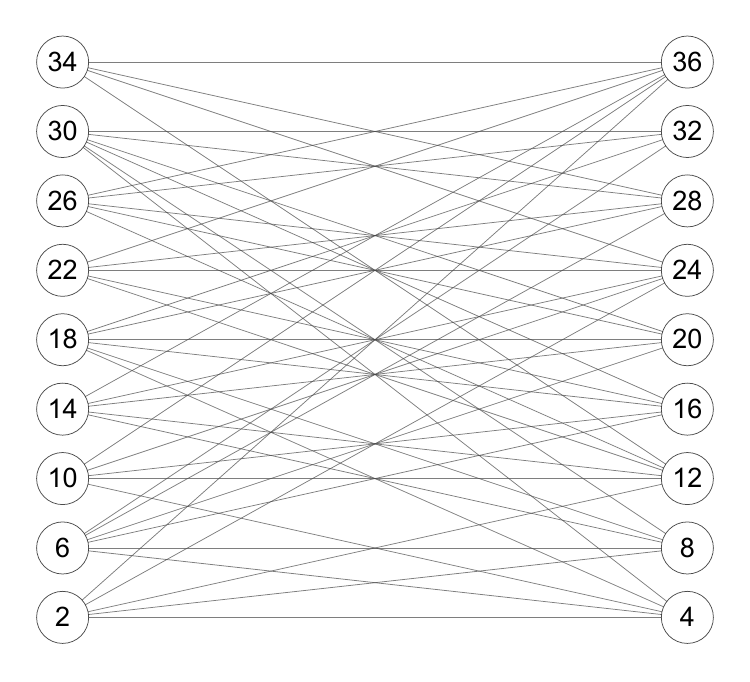}\quad \includegraphics[scale=0.45]{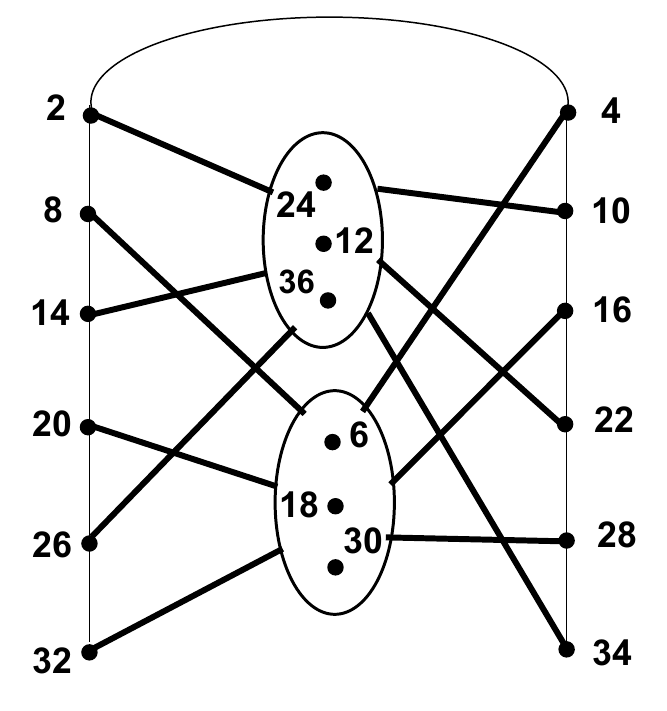}
\caption{The prime multiple missing graph $G(3,18)$ in Example \ref{ex1:318}}\label{gd318}
\end{center}
\end{figure}

\begin{proposition}\label{thmgind}
Let $p$ be an odd prime and $m,n\in\Nat$ such that $m<n$. Then the graph $G(p,m)$ is an induced subgraph of the graph $G(p,n)$.
\end{proposition}

\begin{proof}
Let $A_m$ and $A_n$ be even sets and $O_m$ and $O_n$ be odd sets for graphs $G(p,m)$ and $G(p,n)$ respectively. Then by Definition \ref{defgpn}, we have $A_m\subset A_n$ and $O_m\subset O_n$. So the vertex set of $G(p,m)$ is a subset of the vertex set of $G(p,n)$. Now for any $a,b\in A_m$, if $a$ and $b$ are adjacent in $G(p,m)$, then $\frac{a+b}{2},\frac{|a-b|}{2}\in O_m\subset O_n$. Thus $a$ and $b$ are also adjacent in $G(p,n)$.

Conversely, for any $a,b\in A_m$, if $a$ and $b$ are adjacent in $G(p,n)$, then $\frac{a+b}{2},\frac{|a-b|}{2}\in O_n$. So both these numbers are not non-trivial multiples of $p$. Since $a,b\in A_m$, we have $a,b\leq 2m$ and $a\neq b$. Thus $\frac{a+b}{2},\frac{|a-b|}{2}\leq 2m-1$. This implies $\frac{a+b}{2},\frac{|a-b|}{2}\in O_m$. Hence $a$ and $b$ are adjacent in $G(p,m)$. Therefore, $G(p,m)$ is an induced subgraph of $G(p,n)$.
\end{proof}

Let $G=(V,E_1)$ and $H=(V,E_2)$ be two graphs with the same vertex set $V$. Then the intersection of graphs $G$ and $H$ is a graph $M=(V,E)$, where $E=E_1\cap E_2$. Now for a fixed $n\in\Nat$, we denote the intersection of graphs $G(p_1,n),G(p_2,n),\ldots,G(p_k,n)$ by $G(p_1,p_2,\ldots,p_k,n)$. 

\begin{definition}
Let $G(n)=(V,E)$ be an odd-even graph with the even set\\ $A=\Set{x\in\mathcal{E}}{2\leq x\leq 2n}$ and the odd set $O=\Set{p}{p\text{ is an odd prime},\ p<2n}\cup\set{1}$.\\
We call the graph $G(n)$, a (finite) {\em near Goldbach graph}.
\end{definition}

\begin{proposition}\label{propngb}
Let $n\in\Nat$, $n\geq 5$. The near Goldbach graph $G(n)$ is the intersection of graphs $G(3,n),G(5,n),\ldots,G(p,n)$ where $p$ is the highest prime such that $p^2<2n$. 
\end{proposition}

\begin{proof}
In $G(n)$, we consider the odd set as the set of odd primes along with $1$. In each $G(p,n)$ we remove all non-trivial multiples of the prime $p$. Thus removing all non-trivial multiples of $3,5,\ldots,p$, where $p$ is the highest prime such that $p^2<2n$ from the odd set, it leaves only the odd primes less than $2n$ and the number $1$. Thus the odd set of $G(n)$ is same as the odd set of $G(3,5,\ldots,p,n)$. Also since $n$ is fixed, vertex sets of them are also same. Thus these two graphs are same.
\end{proof}

For example, $G(n)=G(3,n)$ for all $n\leq 12$, $G(n)=G(3,5,n)$ for all $n\leq 24$. In general, $G(n)=G(3,5,\ldots,p_i,n)$ for all $n\leq \frac{p_{i+1}^2-1}{2}$, where $p_{i+1}$ is the prime number next to $p_i$. 

\section{The structure of $G(3,n)$}\label{ssstr}

The first theorem describes the structure of graphs $G(3,n)$.

\begin{theorem}\label{thmg3n}
The graph $G(3,n)$ $(n\geq 6)$ is a bipartite graph and consists of an independent set and a path $P$ and vertices of $P$ are alternatively adjacent to all members of the independent set those belong to the opposite partite sets. Moreover, $G(3,1)\cong K_1$, $G(3,2)\cong K_2$, $G(3,3)\cong P_3$, $G(3,4)\cong C_4$ and $G(3,5)$ consists of a $4$-cycle and a pendant vertex to one vertex of the cycle.
\end{theorem}

\begin{proof}
Let $n\in\Nat$ and $G=G(3,n)=(V,E)$. Then $V=\Set{2x}{x\in\Nat,\ x\leq n}$ and two vertices $a,b\in V$ are adjacent if and only if  $\frac{a+b}{2}$ and $\frac{|a-b|}{2}$ both are odd positive integers but not a non-trivial multiple of $3$.

Let $V_1=\Set{x\in V}{\Div{x}{12}}=\set{12,24,36,\ldots}$, $V_2=\Set{x\in V}{\Div{x}{6},\ \nDiv{x}{12}}$\\
$=\set{6,18,30,\ldots}$ and 
$V_3=V\setminus (V_1\cup V_2)=\set{2,4,8,10,14,16,20,22,26,\ldots}$.

First note that $V_1=\emptyset$ for $n<6$ and $V_2=\emptyset$ for $n<3$. So leaving apart some initial examples, for the sequel, we assume $n\geq 6$. By Remark \ref{rem1}, $G$ is bipartite with partite sets $X$ and $Y$. Thus both $V_1$ and $V_2$ are independent sets as each of them is a subset of a partite set ($V_1\subset X$ and $V_2\subset Y$). Also for any $a\in V_1$ and $b\in V_2$, $\frac{a+b}{2}$ is a non-trivial multiple of $3$. Thus $V_1\cup V_2$ is an independent set.

Now consider the set $V_{31}=\set{2,8,14,20,\ldots}=\Set{6x+2\in V_3}{x\in\Nat\cup\set{0}}$. We note that any two consecutive members of the set are adjacent in $G$ as $\frac{(6x+2)+(6(x+1)+2)}{2}=6x+5$ and $\frac{(6(x+1)+2)-(6x+2)}{2}=3$. Both are not a non-trivial multiple of $3$. Now if we consider any two (distinct) elements $6x+2$ and $6y+2$ in this set, then $\frac{|(6x+2)-(6y+2)|}{2}=3|x-y|$, which is a non-trivial multiple of $3$ unless $|x-y|=1$. Thus this set of elements induce a path in $G$. Similarly, one can show that the set $V_{32}=\set{4,10,16,22,28,\ldots}=\Set{6x+4\in V_3}{x\in\Nat\cup\set{0}}$ also induces a path in $G$. These two paths are joined by the edge between $2$ and $4$ (as they are adjacent) but no other pair of vertices in these two paths are adjacent. Indeed, if we consider two numbers $6x+2$ and $6y+4$, then $\frac{(6x+2)+(6y+4)}{2}=3(x+y)+3$ which is a non-trivial multiple of $3$ unless $x=y=0$. Therefore $V_3$ induces a path, say $P$, in $G$.

Finally, we show that every element of $V_3\cap X$ is adjacent to all members of $V_2$ and every element of $V_3\cap Y$ is adjacent to all members of $V_1$, i.e., vertices in the path $P$ are alternatively adjacent to all members of $V_1$ and $V_2$. If we consider $6x+2$ and $6y+6$, then $\frac{(6x+2)+(6y+6)}{2}=3(x+y+1)+1$ and $\frac{|(6x+2)-(6y+6)|}{2}=|3(x-y)-2|$. If $x$ and $y$ are of opposite parity, then both numbers odd, but not a multiple of $3$. Thus these numbers are adjacent in $G$. So every element of $V_{31}\cap X$ is adjacent to all of $V_2$ and every element of $V_{31}\cap Y$ is adjacent to all of $V_1$. Similarly, we can show that $V_{32}\cap X$ is adjacent to all of $V_2$ and every element of $V_{32}\cap Y$ is adjacent to all of $V_1$.

The small graphs (for $n<6$) are induced subgraphs of $G(3,6)$ by Theorem \ref{thmgind}. Thus they can easily be obtained from Figure \ref{gd318} (right).
\end{proof}

Now let us see the pattern of the adjacency matrix of the graph $G=G(3,n)$. We first note that the graph $G(3,n)$ is a bipartite graph, say, $G=(X,Y,E)$. Since there are no edges between two vertices of the same partite set ($X$ or $Y$), the adjacency matrix $A(G)$ takes the following form:
{\footnotesize $$A(G)\ =\ \begin{array}{c|c|c|}
\multicolumn{1}{c}{} & X & \multicolumn{1}{c}{Y}\\ 
\cline{2-3}
X & \mathbf{0} & B \\ 
\cline{2-3}
Y & B^T & \mathbf{0}\\
\cline{2-3}
\end{array}$$}
where $B$ is the {\em biadjacency matrix} of $G$ and $B^T$ is the transpose of $B$. Thus in order to understand the adjacency matrix structure of a bipartite graph $G$, it is sufficient to study the pattern of the biadjacency matrix of $G$.

\begin{definition}
Let $P=(X,Y,E)$ be a graph which is a path. If we arrange vertices along the path alternatively in rows and columns, then the biadjacency matrix of $P$ is called a {\em path matrix} (as the right bottom submatrix in Table \ref{tabg318}).
\end{definition}

\begin{table}[t]
{\footnotesize $$\begin{array}{c|ccc|cccccc|}
\multicolumn{1}{c}{} & 12 & 24 & \multicolumn{1}{c}{36} & 28 & 16 & \multicolumn{1}{c}{4} & 8 & 20 & \multicolumn{1}{c}{32}\\ 
\cline{2-10}
 6 & 0 & 0 & 0 & 1 & 1 & 1 & 1 & 1 & 1 \\
 18 & 0 & 0 & 0 & 1 & 1 & 1 & 1 & 1 & 1 \\
 30 & 0 & 0 & 0 & 1 & 1 & 1 & 1 & 1 & 1 \\
\cline{2-10}
 34 & 1 & 1 & 1 & 1 & 0 & 0 & 0 & 0 & 0 \\
 22 & 1 & 1 & 1 & 1 & 1 & 0 & 0 & 0 & 0 \\
 10 & 1 & 1 & 1 & 0 & 1 & 1 & 0 & 0 & 0 \\
 2 & 1 & 1 & 1 & 0 & 0 & 1 & 1 & 0 & 0 \\
 14 & 1 & 1 & 1 & 0 & 0 & 0 & 1 & 1 & 0 \\
 26 & 1 & 1 & 1 & 0 & 0 & 0 & 0 & 1 & 1 \\
\cline{2-10}
\end{array}$$}
\caption{The biadjacency matrix of $G(3,18)$ (see Example \ref{ex1:318})}\label{tabg318}
\end{table}

\begin{theorem}
Let $G=G(3,n)$ for some $n\geq 6$. Then the vertices of $G$ can be arranged in such a way that $B(G)$, the biadjacency matrix of $G$ takes the following form, where $A_1$ is a path matrix.
$$\begin{array}{|c|c|}
\hline
\mathbf{0} & \mathbf{1} \\
\hline
\mathbf{1} & A_1 \\
\hline
\end{array}$$
\end{theorem}

\begin{proof}
The graph $G=G(3,n)=(V,E)$ is a bipartite graph with partite sets $X$ and $Y$ as in Remark \ref{rem1}. Now we further split $X$ and $Y$ as follows. Let $X_i=\Set{x\in X}{\Mod{x}{i}{3}}$ and $Y_i=\Set{y\in Y}{\Mod{y}{i}{3}}$. Then we show that the biadjacency matrix $B(G)$ of $G$ takes the prescribed form if we arrange vertices of $X_0$, vertices of $X_1$ in decreasing order and then vertices of $X_2$ in increasing order in rows and similarly arrange vertices of $Y$ in columns. Then by Theorem \ref{thmg3n}, the submatrix $X_1\cup X_2$ by $Y_1\cup Y_2$ is the biadjacency matrix of a path. Also the submatrix $X_0$ by $Y_0$ is a null matrix as $X_0\cup Y_0$ is an independent set. Finally, we have all vertices of $X_0$ ($Y_0$) are adjacent to all vertices of the path which belong to $Y$ ($X$ resp.). So all entries of submatrices $X_0$ by $Y_1\cup Y_2$ and $Y_0$ by $X_1\cup X_2$ are $1$.
\end{proof}

\section{Properties of $G(3,n)$}\label{secg3np}

The structure theorem (Theorem \ref{thmg3n}) suggests that graphs $G(3,n)$ are connected for any $n\in\Nat$. Actually we get something more.

\begin{theorem}\label{thmg3nd}
The graphs $G(3,n)$ are connected for all $n\in\Nat$ with diameter at most $3$.
\end{theorem}

\begin{proof}
We first note that, by Theorem \ref{thmg3n}, diameter of $G(3,n)$ is at most $3$ for $n<6$. Let $n\geq 6$. Then $V_1,V_2,V_3$  all are nonempty as defined in the proof of Theorem \ref{thmg3n} and let $P$ be the path induced by $V_3$. Let $x\in V_1$. Then for any element on $P$ or its neighbor is adjacent to $x$. Since the same is true for any $y\in V_2$, every vertex is at most $3$ distance apart from $x$. The result is same for any vertex in $V_2$. Now take any two non-adjacent vertices on the path $P$. If both of them are adjacent to all vertices of $V_1$, then they are $2$ distant apart. If one, say $x$, is adjacent to all vertices of $V_1$ and the other is adjacent to all vertices of $V_2$, then any neighbor of $x$ on $P$ is adjacent to all vertices of $V_2$. So they are $3$ distance apart. Therefore the graphs $G(3,n)$ are connected and diameters of the graphs are at most $3$.
\end{proof}

Now we observe interesting cycle structure of graphs $G(3,n)$. We know that the {\em girth} of a graph is the length of its shortest cycle.

\begin{proposition}\label{lemcy1}
The girth of graphs $G(3,n)$ is $4$ for all $n\geq 4$.
\end{proposition}

\begin{proof}
The graphs $G(3,n)$ are simple and bipartite. So there are no $1,2$ or $3$-cycles in these graphs. We note that $(2,4,6,8,2)$ is a $4$-cycle in $G(3,n)$ for all $n\geq 4$. Thus the girth is $4$.
\end{proof}

Now apart from the $4$-cycle formed by first $4$ vertices, there are many $4$-cycles in graphs $G(3,n)$.

\begin{lemma}\label{lemcy2}
Every vertex of $G(3,n)$ is a vertex of a $4$-cycle for all $n\geq 6$.
\end{lemma}

\begin{proof}
We first note that, since $n\geq 6$, we have both $V_1$ and $V_2$ are nonempty and the path, say $P$, induced by $V_3$ is of length at least $4$. Let $x,y,z$ be any three consecutive vertices on $P$, then either both $x$ and $z$ are adjacent to all vertices of $V_1$ or they are adjacent to all vertices of $V_2$ (see Theorem \ref{thmg3n}). Thus $(u,x,y,z,u)$ is a $4$-cycle where either $u\in V_1$ or $u\in V_2$. Since $n\geq 6$, we get both cases for first $3$ or last $3$ vertices on any length $4$ subpath of $P$. Thus every vertex is a vertex of a $4$-cycle.
\end{proof}

Now Lemma \ref{lemcy2} gives rise to a question that whether $G(3,n)$ is chordal bipartite\footnote{A bipartite graph is called {\em chordal bipartite} if it does not contain an induced cycle of length greater than $4$.} or not. The following result shows the negative answer for $n\geq 8$.

\begin{lemma}\label{lemcy3}
The graph $G(3,n)$ contains induced $6$-cycles for any $n\geq 8$. 
\end{lemma}

\begin{proof}
We note that $(2,8,6,16,10,12,2)$ is an induced $6$-cycle in $G(3,8)$ as $2\not\leftrightarrow 16$, $8\not\leftrightarrow 10$ and $6\not\leftrightarrow 12$. By Proposition \ref{thmgind}, this is also an induced cycle of graphs $G(3,n)$ for all $n\geq 8$.
\end{proof}

This $6$-cycle is not a single occurrence, one can see that many $6$-cycles can be formed by suitably choosing two disjoint edges of the path $P$ and two vertices, one each from $V_1$ and $V_2$. But this length cannot be made higher.

\begin{proposition}\label{lemcy5}
The graph $G(3,n)$ does not contain an induced $r$-cycle for any $r>6$. 
\end{proposition}

\begin{proof}
We have $V_1\cup V_2$ is an independent set in $G(3,n)$ and for any three consecutive vertices in the path $P$ induced by $V_3$, the first and the last are adjacent to either all vertices of $V_1$ or all vertices of $V_2$. Now every vertex on $P$ is either adjacent to all vertices of $V_1$ or adjacent to all vertices of $V_2$. Thus if we have more than $4$ vertices from $P$ in the cycle, then the degree of a vertex in $V_1$ or $V_2$ would be more than $2$ in the cycle, which is a contradiction. Also one cannot form a cycle only with vertices of $P$. Suppose we have only two vertices on $P$ which are adjacent to a vertex in $V_1$. Now these two vertices are adjacent to at most $4$ more vertices on the path. Thus, in order to form a cycle with more than $6$ vertices, we need to include vertices from $V_2$. But then it would lead to the contradiction stated above. Similar contradiction arises if we interchange the role of $V_1$ and $V_2$ in the above argument. Thus any cycle of length more than $6$ must have a chord and there are no induced cycles of length greater than $6$. Since there are no odd cycles, the result follows.
\end{proof}

\begin{theorem}\label{hamg3n}
The graphs $G(3,n)$ have Hamiltonian paths for all $n\in\Nat$ and they are Hamiltonian for all even $n>2$.
\end{theorem}

\begin{proof}
The graph $G(3,n)$ is itself a path for each $n\leq 3$. Next we note that from any Hamiltonian cycle for the graph $G(3,n)$ $(n\geq 4)$, if we drop the vertex $2n$, it will be a Hamiltonian path for the graph $G(3,n-1)$, as the graph $G(3,n-1)$ is an induced subgraph of $G(3,n)$ by Theorem \ref{thmgind}. Thus it is sufficient to show Hamiltonian cycles of graphs for even $n\geq 4$.

We define $V_1,V_2,V_3,V_{31},V_{32}$ as in the proof of Theorem \ref{thmg3n}. Let us write\\ $V_{31}=\set{2,8,14,20,26,\ldots}=\set{x_1,x_2,x_3,\ldots}$ and $V_{32}=\set{4,10,16,22,28,\ldots}=\set{y_1,y_2,y_3,\ldots}$. Also let $V_1=\set{12,24,36,\ldots}=\set{u_1,u_2,u_3,\ldots}$ and 
$V_2=\set{6,18,30,\ldots}=\set{v_1,v_2,v_3,\ldots}$, arranging vertices in the increasing order of numerical values for each set. We will see the ordering of elements on the path induced by $V_3$ are important to preserve whereas ordering of elements of $V_1$ and $V_2$ are not required as every vertex of $V_3$ is either adjacent to all vertices of $V_1$ or all vertices of $V_2$ according to their belonging in partite sets. We need only their size. Thus to understand the pattern of following cycles we use {\boldmath $U$} and $\text{{\boldmath $V$}}$ to represent an element of $V_1$ and $V_2$ respectively. Now, as $n$ increases, the appearance of a new element follows the sequence below.
\begin{equation}\label{eqap}
V = \set{2,4,6,8,10,\ldots}=\set{x_1,y_1,v_1,x_2,y_2,u_1,x_3,y_3,v_2,x_4,y_4,u_2,\ldots}
\end{equation}
The following are Hamiltonian cycles $C_n$ of $G(3,n)$ for $n\geq 4$.\\[0.5em]
\noindent 
{\footnotesize $n=4$, $V=\set{x_1,y_1,v_1,x_2}$, $C_4=(x_1,x_2,\text{{\boldmath $V$}},y_1,x_1)$.\\[1em]
$n=6$, $V=\set{x_1,y_1,v_1,x_2,y_2,u_1}$, $C_6=(x_1,x_2,\text{{\boldmath $V$}},y_1,y_2,\text{{\boldmath $U$}},x_1)$.\\
$n=8$, $V=\set{x_1,y_1,v_1,x_2,y_2,u_1,x_3,y_3}$, $C_8=(x_1,x_2,x_3,\text{{\boldmath $U$}},y_2,y_3,\text{{\boldmath $V$}},y_1,x_1)$.\\
$n=10$, $V=\set{x_1,y_1,v_1,x_2,y_2,u_1,x_3,y_3,v_2,x_4}$, $C_{10}=(x_1,x_2,\text{{\boldmath $V$}},x_4,x_3,\text{{\boldmath $U$}},y_2,y_3,\text{{\boldmath $V$}},y_1,x_1)$.\\[1em]
$n=12$, $V=\set{x_1,y_1,v_1,x_2,y_2,u_1,x_3,y_3,v_2,x_4,y_4,u_2}$, $C_{12}=(x_1,x_2,\text{{\boldmath $V$}},x_4,x_3,\text{{\boldmath $U$}},y_4,y_3,\text{{\boldmath $V$}},y_1,y_2,\text{{\boldmath $U$}},x_1)$.\\
$n=14$, $V=\set{x_1,y_1,v_1,x_2,y_2,u_1,x_3,y_3,v_2,x_4,y_4,u_2,x_5,y_5}$,\\
\null\hfill $C_{14}=(x_1,x_2,x_3,\text{{\boldmath $U$}},x_5,x_4,\text{{\boldmath $V$}},y_5,y_4,\text{{\boldmath $U$}},y_2,y_3,\text{{\boldmath $V$}},y_1,x_1)$.\\
$n=16$, $V=\set{x_1,y_1,v_1,x_2,y_2,u_1,x_3,y_3,v_2,x_4,y_4,u_2,x_5,y_5,v_3,x_6}$,\\
\null\hfill $C_{16}=(x_1,x_2,\text{{\boldmath $V$}},x_4,x_3,\text{{\boldmath $U$}},x_5,x_6,\text{{\boldmath $V$}},y_5,y_4,\text{{\boldmath $U$}},y_2,y_3,\text{{\boldmath $V$}},y_1,x_1)$.}

Thus we can follow the above pattern which is same for $\Mod{n}{i}{3}$, $i=0,2,1$. Note that along these cycles we have used all the vertices to make it a spanning cycle and the adjacencies are following the rule for graphs $G(3,n)$. In the following, we provide the complete list of general construction of Hamiltonian cycles for all $n\geq 4$, ($n\in\Nat$) 

\vspace{1em}
\noindent
{\bf Case I:}\ $n=6k$, $k\in\Nat$, $k$ is odd.

\vspace{1em}
\noindent
$(x_1,x_2,\ 6,\ x_4,x_3,\ 12,\ x_5,x_6,\ 18,\ldots,$\\
 $x_{2k-5},x_{2k-4},\ 6(k-2),\ x_{2k-2},x_{2k-3},\ 6(k-1),\ x_{2k-1},x_{2k},\ 6k,$\\
$y_{2k-1},y_{2k},\ 6(k+1), \ y_{2k-2},y_{2k-3},\ 6(k+2),\ y_{2k-5},y_{2k-4},\ 6(k+3),$\\
$\ldots, 6(2k-3),\ y_5,y_6,\ 6(2k-2),\ y_4,y_3,\ 6(2k-1),\ y_1,y_2,\ 6(2k),\ x_1)$.

\vspace{1em}
\noindent
{\bf Case II:}\ $n=6k$, $k\in\Nat$, $k$ is even.

\vspace{1em}
\noindent
$(x_1,x_2,\ 6,\ x_4,x_3,\ 12,\ x_5,x_6,\ 18,\ldots,$\\
 $x_{2k-4},x_{2k-5},\ 6(k-2),\ x_{2k-3},x_{2k-2},\ 6(k-1),\ x_{2k},x_{2k-1},\ 6k,$\\
$y_{2k},y_{2k-1},\ 6(k+1), \ y_{2k-3},y_{2k-2},\ 6(k+2),\ y_{2k-4},y_{2k-5},\ 6(k+3),$\\
$\ldots, 6(2k-3),\ y_5,y_6,\ 6(2k-2),\ y_4,y_3,\ 6(2k-1),\ y_1,y_2,\ 6(2k),\ x_1)$.

\vspace{1em}
\noindent
Note that, in both the above cases, $|V_1|=|V_2|=k$, $|V_{31}|=|V_{32}|=2k$. 

\vspace{1em}
\noindent
{\bf Case III:}\ $n=6k+2$, $k\in\Nat$, $k$ is odd.

\vspace{1em}
\noindent
$(x_1,x_2,x_3\ 12,\ x_5,x_4,\ 6,\ x_6,x_7,\ 24,\ x_{9},x_8,\ 18,\ldots,$\\
 $x_{2k-4},x_{2k-3},\ 6(k-1),\ x_{2k-1},x_{2k-2},\ 6(k-2),\ x_{2k},x_{2k+1},\ 6(k+1),$\\
$y_{2k},y_{2k+1},\ 6(k), \ y_{2k-1},y_{2k-2},\ 6(k+3),\ y_{2k-4},y_{2k-3},\ 6(k+2),$\\
$\ldots, 6(2k-2),\ y_6,y_7,\ 6(2k-3),\ y_5,y_4,\ 6(2k),\ y_2,y_3,\ 6(2k-1),\ y_1,x_1)$.

\vspace{1em}
\noindent
{\bf Case IV:}\ $n=6k+2$, $k\in\Nat$, $k$ is even.

\vspace{1em}
\noindent
$(x_1,x_2,x_3\ 12,\ x_5,x_4,\ 6,\ x_6,x_7,\ 24,\ x_{9},x_8,\ 18,\ldots,$\\
 $x_{2k-3},x_{2k-4},\ 6(k-3),\ x_{2k-2},x_{2k-1},\ 6(k),\ x_{2k+1},x_{2k},\ 6(k-1),$\\
$y_{2k+1},y_{2k},\ 6(k+2), \ y_{2k-2},y_{2k-1},\ 6(k+1),\ y_{2k-3},y_{2k-4},\ 6(k+4),$\\
$\ldots, 6(2k-2),\ y_6,y_7,\ 6(2k-3),\ y_5,y_4,\ 6(2k),\ y_2,y_3,\ 6(2k-1),\ y_1,x_1)$.

\vspace{1em}
\noindent
Note that, in both the above cases, $|V_1|=|V_2|=k$, $|V_{31}|=|V_{32}|=2k+1$. 

\vspace{1em}
\noindent
{\bf Case V:}\ $n=6k+4$, $k\in\Nat$, $k$ is odd.

\vspace{1em}
\noindent
$(x_1,x_2,\ 6,\ x_4,x_3,\ 12,\ x_5,x_6,\ 18,\ldots,$\\
 $x_{2k-5},x_{2k-4},6(k-2),x_{2k-2},x_{2k-3},6(k-1),x_{2k-1},x_{2k},6k,x_{2k+2},x_{2k+1},6(k+1)$\\
$y_{2k},y_{2k+1},\ 6(k+2), \ y_{2k-1},y_{2k-2},\ 6(k+3),\ y_{2k-4},y_{2k-3},\ 6(k+4),$\\
$\ldots, 6(2k-2),\ y_6,y_7,\ 6(2k-1),\ y_5,y_4,\ 6(2k),\ y_2,y_3,\ 6(2k+1),\ y_1,x_1)$.

\vspace{1em}
\noindent
{\bf Case VI:}\ $n=6k+4$, $k\in\Nat$, $k$ is even.

\vspace{1em}
\noindent
$(x_1,x_2,\ 6,\ x_4,x_3,\ 12,\ x_5,x_6,\ 18,\ldots,$\\
 $x_{2k-4},x_{2k-5},6(k-2),x_{2k-3},x_{2k-2},6(k-1),x_{2k},x_{2k-1},6k,x_{2k+1},x_{2k+2},6(k+1)$\\
$y_{2k+1},y_{2k},\ 6(k+2), \ y_{2k-2},y_{2k-1},\ 6(k+3),\ y_{2k-3},y_{2k-4},\ 6(k+4),$\\
$\ldots, 6(2k-2),\ y_6,y_7,\ 6(2k-1),\ y_5,y_4,\ 6(2k),\ y_2,y_3,\ 6(2k+1),\ y_1,x_1)$.

\vspace{1em}
\noindent
Note that, in both the above cases, $|V_1|=k$, $|V_2|=k+1$, $|V_{31}|=2k+2$, $|V_{32}|=2k+1$. 

\vspace{1em}
\noindent
Following above formulas, Hamiltonian cycles of $G(3,n)$ for $4\leq n\leq 40$ are provided in Table \ref{thamcycleg3n}.
\end{proof}

\section{The graphs $G(5,n)$}\label{secg5n}

In this section we briefly describe the structure, adjacency matrix pattern and some properties of graphs $G(5,n)$. We begin with an example.

\begin{example}
Let us consider the graph $G(5,30)=(V,E)$. Then the even set, $A=\set{2,4,6,\ldots,60}$ and the odd set $O=\set{1,3,5,7,9,11,13,17,19,21,23,27,29,31,33,37,39,41,43,47,49,51,53,57,59}$.\\ Let $a,b\in A$. Then $a$ and $b$ are adjacent if and only if $\frac{a+b}{2},\frac{|a-b|}{2}\in O$. The graph is bipartite with partite sets $X$ and $Y$ (see Remark \ref{rem1}). It consists of two paths and an independent set. If we separate alternate elements of those paths, then the set of vertices of each path splits into two parts (belonging to two partite sets) and they are adjacent to all vertices of the independent set as well as all vertices of the other path those belong to the opposite partite set. In Figure \ref{gd530}, bold lines between two sets indicate that all vertices of one set are adjacent to all vertices of the other. The Table \ref{tgd530} shows the biadjacency matrix of $G(5,30)$.
\end{example}

\begin{figure}[t]
\begin{center}
\includegraphics[scale=0.4]{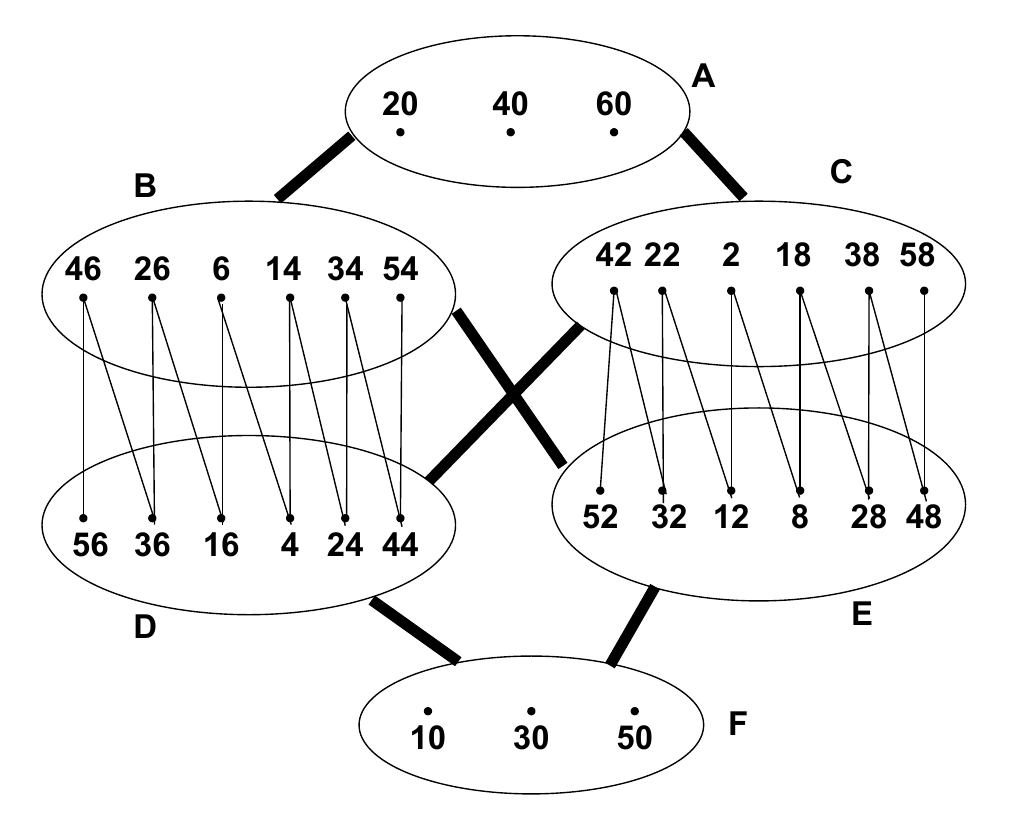}
\caption{The prime multiple missing graph $G(5,30)$}\label{gd530}
\end{center}
\end{figure}

\begin{theorem}\label{thmg5n}
The graph $G(5,n)$ $(n\geq 10)$ is a bipartite graph and consists of an independent set and two paths. The set of vertices of each path splits into two parts (belonging to two partite sets) and they are adjacent to all vertices of the independent set as well as all vertices of the other path those belong to the opposite partite set. Moreover, $G(5,1)\cong K_1$, $G(5,2)\cong K_2$, $G(5,3)\cong P_3$, $G(5,4)\cong C_4$ and $G(5,5)\cong K_{2,3}$, $G(5,6)\cong K_{3,3}$, $G(5,7)\cong K_{3,4}$, $G(5,8)$ is the graph obtained from $K_{4,4}$ by deleting exactly one edge and $G(5,9)$ is obtained from $K_{4,5}$ by deleting two disjoint edges.
\end{theorem}

\begin{proof}
Let $G=G(5,n)=(V,E)$ for $n\geq 10$. Let $X=\Set{x\in V}{\Mod{x}{0}{4}}$\\ 
$=A\cup D\cup E$ and $Y=\Set{y\in V}{\Mod{y}{2}{4}}=F\cup B\cup C$,\\ where $A=\Set{x\in X}{\Mod{x}{0}{5}}$,\\ $D=\Set{x\in X}{\Mod{x}{1}{5}}\cup \Set{x\in X}{\Mod{x}{4}{5}}$ and\\ $E=\Set{x\in X}{\Mod{x}{2}{5}}\cup \Set{x\in X}{\Mod{x}{3}{5}}$. The sets $F,B,C$ are defined similarly as subsets of $Y$. Note that, since $n\geq 10$, all these sets $A,B,C,D,E,F$ are nonempty.

We have $G$ is a bipartite graph with partite sets $X$ and $Y$. For any $x\in X$ and $y\in Y$, $x$ is adjacent to $y$ if and only if $\frac{a+b}{2}$ and $\frac{|a-b|}{2}$ are not a non-trivial multiple of $5$. Thus it follows that the set $A\cup F$ is an independent set. So $A$ and $F$ are independent sets and there are no edges between them. Now consider the set $B\cup D$. First we have $4$ and $6$ are adjacent as $\frac{4+6}{2}=5$ and $\frac{|4-6|}{2}=1$. Except this edge all other edges have end vertices whose numerical difference is $10$. No other vertices are adjacent here as either $\frac{a+b}{2}$ or $\frac{|a-b|}{2}$ would be a non-trivial multiple of $5$. Thus the vertices in $B\cup D$ form a path. Similarly, vertices in the set $C\cup E$ form a path.

Now any element of $B$ is either $4$ modulo $5$ or $1$ modulo $5$. Thus $\frac{a+b}{2}$ and $\frac{|a-b|}{2}$ can never be a multiple of $5$ for any $b\in B$ and $a\in A$ as all elements of $A$ are $0$ modulo $5$. Therefore all elements of $B$ are adjacent with every element of $A$. Also elements of $E$ are either $2$ modulo $5$ or $3$ modulo $5$. Thus elements of $B$ are also adjacent with every element of $E$. Similar connections can be established between $\set{A,C}$, $\set{D,C}$, $\set{D,F}$ and $\set{E,F}$. Thus we have the structure of graphs $G(5,n)$ as described in the statement.
 As in the proof of Theorem \ref{thmg3n}, small graphs can be obtained as induced subgraphs of $G(5,10)$ (see Figure \ref{gd530}).
\end{proof}

\begin{theorem}
Let $G=G(5,n)$. Then there is an ordering of vertices of $G$ such that the biadjacency matrix of $G$ takes the following form:
{\footnotesize $$\begin{array}{c|c|c|c|}
\multicolumn{1}{c}{} & \multicolumn{1}{c}{A} & \multicolumn{1}{c}{D} & \multicolumn{1}{c}{E}\\
\cline{2-4}
F & \mathbf{0} & \mathbf{1} & \mathbf{1}\\
\cline{2-4}
B & \mathbf{1} & A_1 & \mathbf{1}\\
\cline{2-4}
C & \mathbf{1} & \mathbf{1} & A_2\\
\cline{2-4}
\end{array}$$}
where $A_1$ and $A_2$ are path matrices. The sets $A,B,C,D,E,F$ are as in the proof of Theorem \ref{thmg5n}.
\end{theorem}

\begin{proof}
The proof follows directly from Theorem \ref{thmg5n} and hence omitted.
\end{proof}

\begin{table}[ht]
{\tiny 
$$\begin{array}{c|ccc|cccccc|cccccc|}
 \multicolumn{1}{c}{} & 20 & 40 & \multicolumn{1}{c}{60} & 56 & 36 & 16 & 4 & 24 & \multicolumn{1}{c}{44} & 52 & 32 & 12 & 8 & 28 & \multicolumn{1}{c}{48} \\
\cline{2-16}
 10 & 0 & 0 & 0 & 1 & 1 & 1 & 1 & 1 & 1 & 1 & 1 & 1 & 1 & 1 & 1 \\
 30 & 0 & 0 & 0 & 1 & 1 & 1 & 1 & 1 & 1 & 1 & 1 & 1 & 1 & 1 & 1 \\
 50 & 0 & 0 & 0 & 1 & 1 & 1 & 1 & 1 & 1 & 1 & 1 & 1 & 1 & 1 & 1 \\
\cline{2-16}
 46 & 1 & 1 & 1 & 1 & 1 & 0 & 0 & 0 & 0 & 1 & 1 & 1 & 1 & 1 & 1 \\
 26 & 1 & 1 & 1 & 0 & 1 & 1 & 0 & 0 & 0 & 1 & 1 & 1 & 1 & 1 & 1 \\
 6 & 1 & 1 & 1 & 0 & 0 & 1 & 1 & 0 & 0 & 1 & 1 & 1 & 1 & 1 & 1 \\
 14 & 1 & 1 & 1 & 0 & 0 & 0 & 1 & 1 & 0 & 1 & 1 & 1 & 1 & 1 & 1 \\
 34 & 1 & 1 & 1 & 0 & 0 & 0 & 0 & 1 & 1 & 1 & 1 & 1 & 1 & 1 & 1 \\
 54 & 1 & 1 & 1 & 0 & 0 & 0 & 0 & 0 & 1 & 1 & 1 & 1 & 1 & 1 & 1 \\
\cline{2-16}
 42 & 1 & 1 & 1 & 1 & 1 & 1 & 1 & 1 & 1 & 1 & 1 & 0 & 0 & 0 & 0 \\
 22 & 1 & 1 & 1 & 1 & 1 & 1 & 1 & 1 & 1 & 0 & 1 & 1 & 0 & 0 & 0 \\
 2 & 1 & 1 & 1 & 1 & 1 & 1 & 1 & 1 & 1 & 0 & 0 & 1 & 1 & 0 & 0 \\
 18 & 1 & 1 & 1 & 1 & 1 & 1 & 1 & 1 & 1 & 0 & 0 & 0 & 1 & 1 & 0 \\
 38 & 1 & 1 & 1 & 1 & 1 & 1 & 1 & 1 & 1 & 0 & 0 & 0 & 0 & 1 & 1 \\
 58 & 1 & 1 & 1 & 1 & 1 & 1 & 1 & 1 & 1 & 0 & 0 & 0 & 0 & 0 & 1 \\
\cline{2-16}
\end{array}$$}
\caption{The biadjacency matrix of the graph $G(5,30)$}\label{tgd530}
\end{table}

\begin{theorem}
The graphs $G(5,n)$ are connected with diameter at most $3$ for all $n\in\Nat$.
\end{theorem}

\begin{proof}
Let $n\geq 6$. We continue the same definitions for the sets $A,B,C,D,E,F$. All these sets are nonempty as $n\geq 6$. First consider the set $B$. Let $x\in B$. Since $x$ is adjacent to all vertices of $A$ and $E$, the distance between $x$ and any one of them is $1$. Also the distance of $x$ from any other vertex in $B$ itself or from any vertex of $C$ is $2$ via any vertex of $A$. Since any vertex of $C$ is adjacent to any vertex of $D$, the distance of $x$ from any vertex of $D$ is at most $3$. Similarly as vertices of $E$ and $F$ are adjacent with each other, the distance of $x$ from any vertex of $F$ is at most $3$. 
 
Next we consider the set $A$. Let $a\in A$. Then the distance of $a$ from any vertex of $B$ or $C$ is $1$, from any vertex of $D$ or $E$ is $2$ and from any vertex of $F$ is $3$. Similarly, we can show that the distance between any two vertices of $G$ is at most $3$. In fact, the diameter is exactly $3$ for $n\geq 6$ as $(10,4,2,20)$ is one of the shortest path between $10$ and $20$. Finally for small graphs the proof follows directly from Theorem \ref{thmg5n}.
\end{proof}

\begin{theorem}\label{thmham5}
The graphs $G(5,n)$ have Hamiltonian paths for all $n\in\Nat$ and they are Hamiltonian for all even $n>2$.
\end{theorem}

\begin{proof}
The graph $G(5,n)$ is itself a path for $n<4$. The graph $G(5,4)$ is a cycle. Now $(2,4,6,8,10)$ is a Hamiltonian path of $G(5,5)$ and $(2,4,6,8,10,12,2)$ is a Hamiltonian cycle in $G(5,6)$. Also $(2,4,6,8,10,12,14)$ is a Hamiltonian path of $G(5,7)$. In the following we exhibit Hamiltonian cycles for $n\in H=\set{8,10,12,14,16}$. Following these we can obtain Hamiltonian cycle for any larger $n$ which is congruent to any element of $H$ modulo $10$ with the same pattern. Also if $G(5,2n)$ $(n>1)$ is Hamiltonian, then deleting the vertex $4n$ we get a Hamiltonian path for $G(5,2n-1)$.\\[1em]
$\begin{array}{ll}
n=8 & (2,4,6,8,14,12,10,16,2)\\
n=10 & (2,4,6,8,10,12,14,20,18,16,2)\\
n=12 & (2,4,6,8,10,12,14,24,22,20,18,16,2)\\
n=14 & (2,4,6,8,10,12,14,24,22,20,18,28,26,16,2)\\
n=16 & (2,4,6,8,10,12,14,24,22,32,30,28,26,16,18,20,2)
\end{array}$\\[1em]
We first note that the following sequence of even integers is a path for graphs $G(5,n)$ (for appropriate $n\geq 8$):
$$(2,4,6,8,10,12,14,\ 24,22,20,18,16,\ 26,28,30,32,34,\ 44,42,40,38,36,\ 46,48,\ldots)$$
That is, after $14$, numbers are divided into blocks of $5$ and they are reversed alternatively. Within each block, for any $2$ consecutive numbers $x,y$, we have $\nMod{x+y}{0}{10}$ and $\frac{|x-y|}{2}=1$. So $x\leftrightarrow y$ in $G(5,n)$. If $x$ is the last number of a block and $y$ is the first number of the next block, then both of them are either $4$ or $6$ modulo $10$. So $\nMod{x+y}{0}{10}$ and $\frac{|x-y|}{2}=5$. Thus $x\leftrightarrow y$ in $G(5,n)$. Therefore, the above sequence is a path in $G(5,n)$, except possibly the last broken block of the sequence that has less than $5$ elements.

Below we describe how to obtain a Hamiltonian cycle of $G(5,n)$ by modifying the above path for $n$ is congruent $2$ or $4$ modulo $10$. 

If $\Mod{n}{2}{10}$, i.e., $n=10m+2$ for some $m\in\Nat$, then we note that $2n=20m+4$. Now the block of this number in the above sequence is reversed as $20m+4,20m+2,20m,20m-2,20m-4$. We note that $\frac{(20m-4)+2}{2}=10m-1$ and $\frac{(20m-4)-2}{2}=10m-3$, both of which are odd and not a multiple of $5$. Thus $20m-4$ is adjacent to $2$. So the last member in the block is adjacent to the first in the above sequence and hence we get a Hamiltonian cycle. For example, \\
{\footnotesize $(2,4,6,8,10,12,14,24,22,20,18,16,26,28,30,32,34,44,42,40,38,36,2)$}
is a Hamiltonian cycle in $G(5,22)$.\\
This is the simplest case. Others require more workout.  

Suppose $\Mod{n}{4}{10}$. Then let $n=10m+4$ for some $m\in\Nat$. Now the block of $2n=20m+8$ consists of only two elements, $20m+6,20m+8$ and its previous block is $20m+4,20m+2,20m,20m-2,20m-4$. We carry on the above sequence before this block as it is. Then we put\\ $20m+4,20m+2,20m,20m-2,20m+8,20m+6,20m-4,2$. We verify that $\frac{(20m-2)+(20m+8)}{2}=20m+3$, $\frac{|(20m-2)-(20m+8)|}{2}=5$, $\frac{(20m+6)+(20m-4)}{2}=20m+1$,\\ $\frac{|(20m+6)-(20m-4)|}{2}=5$ and $\frac{(20m-4)+2}{2}=10m-1$, $\frac{|(20m-4)-2|}{2}=10m-3$, none of which is a non-trivial multiple of $5$. Thus we get a Hamiltonian cycle. For example, the following is a Hamiltonian cycle in $G(5,24)$:\\
$(2,4,6,8,10,12,14,24,22,20,18,16,26,28,30,32,34,44,42,40,38,48,46,36,2)$.

Now we provide the general formula for constructing Hamiltonian cycles of $G(5,n)$.

Let $\Mod{n}{i}{10}$ and $n=10m+i$ for some $m\in\Nat$, $m\geq 2$. We note that the above path can be written as\\[1em]
$(2,4,6,8,10,12,14,\ 24,22,20,18,16,\ldots, 20m-14,20m-12,20m-10,20m-8,20m-6,$\\
$\ 20m+4,20m+2,20m,20m-2,20m-4,\ 20m+6,20m+8,20m+10,20m+12,20m+14, $\\
$\ 20m+16,\ldots)$.\\[1em]
Now $2n=20m+2i$ and it lies in a reversed block for $i=0,2,8$ and belongs to a straight block for $i=4,6$. Now for all $n\geq 20$, we present the general construction for Hamiltonian cycles of graphs $G(5,n)$, where $\Mod{n}{i}{10}$. One may verify that following are Hamiltonian cycle as we did in two above cases. \\[0.5em]

$i=0:$\ $(2,4,6,8,10,12,14,\ 24,22,20,18,16,\ldots, 20m-6,\ 20m,20m-2, 20m-4,2)$\\

$i=2:$\ $(2,4,6,8,10,12,14,\ 24,22,20,18,16,\ldots, 20m-6,\ 20m+4,$\\
$\null\hfill 20m+2,20m,20m-2,20m-4,2)$\\

$i=4:$\ $(2,4,6,8,10,12,14,\ 24,22,20,18,16,\ldots, 20m-6,\ 20m+4,$\\
$\null\hfill 20m+2,20m,20m-2,20m+8,20m+6,20m-4,2)$\\

$i=6:$\ $(2,4,6,8,10,12,14,\ 24,22,20,18,16,\ldots, 20m-6,\ 20m+4,$\\
$\null\hfill 20m+2,20m+12,20m+10,20m+8,20m+6,20m-4,20m-2,20m,2)$\\

$i=8:$\ $(2,4,6,8,10,12,14,\ 24,22,20,18,16,\ldots, 20m-4,\ 20m+6,$\\
$\null\hfill 20m+8,20m+14,20m+12,20m+10,20m+16,2)$.

\vspace{1em}
\noindent
Following above formulas, Hamiltonian cycles of $G(5,n)$ for $18\leq n\leq 38$ are given in Table \ref{thamcycleg5n}.
\end{proof}

\section{The graphs $G(3,n)\cap G(5,n)$}\label{secg35n}

In order to get an idea of near Goldbach graphs, it is very much important to study intersections of prime multiple missing graphs. We begin with the graph $G(3,5,n)=G(3,n)\cap G(5,n)$. 

\begin{theorem}\label{thmg35n}
Let $n\in\Nat$, $n>12$. The graph $G(3,5,n)=(V,E)$ is a bipartite graph with partite sets $X$ and $Y$ and consists of an independent set $V_1\cup V_2$ $(V_1\subseteq X$, $V_2\subseteq Y)$ and some disjoint paths. One path is isomorphic to $P_{10}$, several others are isomorphic to $P_5$ and one or two paths are of length less than $5$. The adjacency between vertices of paths are alternative for vertices of the independent set, belonging to opposite partite set, provided their sum or difference are not a non-trivial multiple of $10$ (see Table \ref{tabv1v2}). Every vertex in the paths is adjacent to some vertices of $V_1$ or $V_2$ and every vertex of $V_1$ and $V_2$ is adjacent to some vertex of the path $P_{10}$. Any graph $G(3,5,n)= G(3,n)$ for $n\leq 12$.
\end{theorem}

\begin{proof}
We first note that, since $5^2=25$ is the first non-trivial multiple of $5$ which is not a multiple of $3$, $G(3,5,n)=G(3,n)$ for all $n\leq 12$. Thus we consider graphs $G(3,5,n)=(V,E)$ for $n>12$. As usual let 
$V_1=\Set{x\in V}{\Mod{x}{0}{12}}$ $=\set{12,24,36,48,\ldots}$, $V_2=\Set{x\in V}{\Mod{x}{6}{12}}$ $=\set{6,18,30,36,\ldots}$ and $V_3=V\smallsetminus (V_1\cup V_2)$. Let $V_{31}=\Set{x\in V_3}{\Mod{x}{2}{3}}$ $=\set{2,8,14,20,\ldots}$ and $V_{32}=\Set{x\in V_3}{\Mod{x}{1}{3}}$ $=\set{4,10,16,22,\ldots}$. Now in the graph $G(3,n)$, vertices in $V_{31}$ induce a path, say $P(1)$ and vertices in $V_{32}$ induce a path, say $P(2)$, where vertices are consecutive according to the increasing order. These two paths are connected by the edge $2\leftrightarrow 4$ and the combined path is denoted by $P$ (see Theorem \ref{thmg3n}).
 
It is clear that the graph $G(3,5,n)$ can be obtained from the graph $G(3,n)$ by deleting edges $xy$, where either $\frac{x+y}{2}$ or $\frac{x-y}{2}$ is a non-trivial multiple of $5$. Now along the path $P$, $\frac{x-y}{2}=3$ for any two consecutive vertices $x,y$, except the case when $\set{x,y}=\set{2,4}$. Thus the only missing edges along $P$ are $\Set{xy\in E(G(3,n))}{\Div{10}{x+y}\text{ and }x+y>10}$. Now along $P(1)$, first such occurrence is $32\leftrightarrow 38$ and that along $P(2)$ is $22\leftrightarrow 28$.\\
We denote the path $(32,26,20,14,8,2,4,10,16,22)$ in $G(3,5,n)$ by $P_0$, which we call the {\em initial path}.\footnote{Note that $\frac{x+y}{2}$ is prime for any two consecutive vertices in $P_0$. Thus the initial path $P_0$ is an induced subgraph for all prime multiple missing graphs, any intersection of them and for any near Goldbach graph.} 
 
Now we notice along $P(1)$, $(38,44,50,56,62)$ is the path next to $32$ and there is another break at the edge $62\leftrightarrow 68$. Since $\Mod{38}{8}{10}$, we have the numbers modulo $10$ along $P(1)$ after $32$ are $\set{8,4,0,6,2,8,4,0,6,2,\ldots}$. Now $8+4=12$, $4+0=4$, $0+6=6$, $6+2=8$ and $2+8=10$. Thus the sum of two consecutive numbers in the sequence is divisible by $10$ only at positions $5k$ and $5k+1$ for all $k\in\Nat$. Therefore, after $32$, $P(1)$ splits into copies of $P_5$, path of $5$ vertices, where the last part is a path of less than or equal to $5$ vertices according to the value of $n$. Also since $\Mod{28}{8}{10}$, the same pattern is followed for $P(2)$. Thus in the graph $G(3,5,n)$, the path $P$ in $G(3,n)$ splits into $P_0$, some copies of $P_5$ and at both ends two other paths of length less than or equal to $4$. For convenience, let us call all such parts, as {\em strips}. 
 
Next let us explore missing edges between vertices of $P$ and vertices in $V_1$ or $V_2$ in $G(3,5,n)$. In $G(3,n)$ every vertex in $P$ are alternatively adjacent to vertices of $V_1$ or $V_2$. Here also some edge $xy$ is missing if and only if either $\frac{x+y}{2}$ or $\frac{x-y}{2}$ is a non-trivial multiple of $5$. First we consider the initial path $P_0$. Note that it is a path of $10$ vertices $(32,26,20,14,8,2,4,10,16,22)$. 
 
Now in $G(3,n)$, vertices $\set{26,14,2,10,22}$ are adjacent to all vertices of $V_1$ and they are\\ $(6,4,2,0,2)$ modulo $10$. We have 
 $V_1=\set{12,24,36,48,60,72,84,96,108,120,\ldots}$ which is\\
 $\{2,4,6,8,0,2,4,6,8,0,\ldots\}$ under modulo $10$. So we have $5$ categories of numbers modulo $10$. Let $x\in V_1$. If $\Mod{x}{2}{10}$ or $\Mod{x}{8}{10}$, then it is adjacent to $3$ vertices $\set{26,14,10}$ of $P_0$. If $\Mod{x}{4}{10}$ or $\Mod{x}{6}{10}$, then it is adjacent to $3$ vertices $\set{2,10,22}$. If $\Mod{x}{0}{10}$, then it is adjacent to $4$ vertices $\set{26,14,2,22}$. So any vertex in $V_1$ is adjacent to at least $3$ vertices of $P_0$. Note that $12$ is also adjacent to $2,22$ and $24$ is adjacent to $14$.

In $G(3,n)$, we have vertices $\set{32,20,8,4,16}$ are adjacent to all vertices of $V_2$. These numbers are $\set{2,0,8,4,6}$ under modulo $10$. Now \\
$V_2=\set{6,18,30,42,54,66,78,90,102,114,\ldots}$ which is $\set{6,8,0,2,4,6,8,0,2,4\ldots}$ under modulo $10$. Let $x\in V_2$. Then $\Mod{x}{2}{10}$ or $\Mod{x}{8}{10}$, then it is adjacent to $3$ vertices $\set{20,4,16}$. If $\Mod{x}{4}{10}$ or $\Mod{x}{6}{10}$, then it is adjacent to $3$ vertices $\set{32,20,8}$. If $\Mod{x}{0}{10}$, then it is adjacent to $4$ vertices $\set{32,8,4,16}$. So any vertex in $V_2$ is adjacent to at least $3$ vertices of $P_0$. Note that $18$ is adjacent to $8$.

Now consider a strip $T$ isomorphic to $P_5$. It is a path of the form $(z,z+6,z+12,z+18,z+24)$ where $\Mod{z}{8}{10}$. Thus these numbers are $\set{8,4,0,6,2}$ modulo $10$. If the first member is adjacent to vertices of $V_1$ in $G(3,n)$, then $3$rd and $5$th are also so whereas $2$nd and $4$th are adjacent to vertices of $V_2$ in $G(3,n)$. Let $x\in V_1$. If $\Mod{x}{2}{10}$ or $\Mod{x}{8}{10}$, then it is adjacent to only the central vertex of $T$. If $\Mod{x}{4}{10}$ or $\Mod{x}{6}{10}$, then it is adjacent to all $1$st, $3$rd and $5$th vertices of $T$. If $\Mod{x}{0}{10}$, then it is adjacent to both end points of $T$. Let $y\in V_2$. If $\Mod{y}{2}{10}$ or $\Mod{y}{8}{10}$ or $\Mod{y}{0}{10}$, then it is adjacent to both $2$nd and $4$th vertices of $T$. If $\Mod{x}{4}{10}$ or $\Mod{x}{6}{10}$, then it is not adjacent to any vertex of $T$. On the other hand, if the first member is adjacent to vertices of $V_2$ in $G(3,n)$, then the role of vertices of $V_1$ and $V_2$ will be interchanged in the statements above (in this paragraph).  
Finally if the strip is smaller than $P_5$, then adjacencies with vertices of $V_1$ and $V_2$ are determined by cases as considered and calculated above. 
\end{proof}

\begin{table}[ht]
{\footnotesize $$\begin{array}{|c|c|c|c|c|}
\hline
 X & i & y\in P_0\cap N(x) & y\in P_5^{(1)}\cap N(x) & y\in P_5^{(2)}\cap N(x)\\
& & & (position) & (position)\\
\hline
V_1 & 2,8 & 26,14,10 & 3 & 2, 4\\
\hline
V_1 & 4,6 & 2,10,22 & 1, 3, 5 & \emptyset\\
\hline
V_1 & 0 & 26,14,2,22 & 1, 5 & 2, 4\\
\hline
V_2 & 2,8 & 20,4,16 & 2, 4 & 3 \\
\hline
V_2 & 4,6 & 32,20,8 & \emptyset & 1, 3, 5\\
\hline
V_2 & 0 & 32,8,4,16 & 2, 4 & 1, 5\\
\hline
\end{array}$$}
\caption{Adjacency of the vertex $x\in X$ where $\Mod{x}{i}{10}$ in $G(3,5,n)$ with the possible exception when $|x-y|=10$. The set $P_5^{(1)}$ ($P_5^{(2)}$) denotes the set of vertices (in increasing order) of a strip $P_5$ where the $1^\text{st}$ vertex is adjacent to vertices of $V_1$ (resp. $V_2$) in $G(3,n)$.}\label{tabv1v2}
\end{table}

Now we show connectedness and determine the diameter of graphs $G(3,5,n)$.

\begin{theorem}
The graphs $G(3,5,n)$ are connected with diameter at most $4$ for any $n\in\Nat$.
\end{theorem}

\begin{proof}
By Theorem \ref{thmg3nd}, the diameter of $G(3,5,n)=G(3,n)$ is at most $3$ for $n\leq 12$. It can be observed that the diameter of $G(3,5,n)$ is $3$ for $n=13$ and is $4$ for $n=14,15$. So let us assume that $n\geq 16$. We divide the set of vertices into several parts. Let $X_1=V_1$, $X_2=V_2$, $X_3=P_0$ (the initial path), $X_4$ be the set of $5$ vertices of a strip isomorphic to $P_5$, $X_5$ be the set of vertices of a strip of length less than $4$.

{\bf Case I \& II:}\\
Let $x,y\in X_1$ or $x,y\in X_2$. From Table \ref{tabv1v2}, we have $x$ and $y$ have a common neighbor in $X_3$. Thus $d(x,y)=2$.

{\bf Case III:}\\
Let $x\in X_1$ and $y\in X_2$. Now either $x$ is adjacent to $2$ or it is adjacent to $10$ and $14$, whereas $y$ is adjacent to $4$ or $8$. In $P_0$, $2$ is adjacent to both $4$ and $8$. Also $10$ is adjacent to $4$ and $14$ is adjacent to $8$. Thus $d(x,y)= 3$. Therefore $d(x,y)\leq 3$ for all $x,y\in V_1\cup V_2$. \footnote{At this point we note that every vertex of $G(3,5,n)$ other than $V_1\cup V_2$ is adjacent to either a vertex of $V_1$ or a vertex of $V_2$ (see Table \ref{tabv1v2}). Since $d(x,y)\leq 3$ for all $x,y\in V_1\cup V_2$, we have the distance between any two vertices of the graph is at most $5$. Thus the diameter of $G(3,5,n)$ is at most $5$.}

{\bf Case IV:}\\
Let $x\in X_1$ and $y\in X_3$. Now there are $3$ options of neighbors of $x$ in $P_0$ which are $\set{26,14,10}$, $\set{2,10,22}$, $\set{26,14,2,22}$. Since $P_0=(32,26,20,14,8,2,4,10,16,22)$, $d(x,y)\leq 2$ in first and third cases. For the second case, $d(x,y)\leq 4$ unless $y=32$ or $26$. Since we have a path $(32,26,12,2,x)$, in this case also $d(x,y)\leq 4$.

{\bf Case V:}\\
Let $x\in X_1$ and $y\in X_4$. If end points of $X_4$ are adjacent to all vertices of $V_1$ in $G(3,n)$, then $x$ is adjacent to $1$st and $5$th vertices of $X_4$ or the $3$rd vertex of $X_4$ (or both). Then $d(x,y)\leq 3$. Now if end points of $X_4$ are adjacent to all vertices of $V_2$ in $G(3,n)$, then either $x$ is adjacent to both $2$nd and $4$th vertices of $X_4$ or it is adjacent to none of them. In the first case $d(x,y)\leq 2$. In the second case $x$ is $4$ or $6$ modulo $10$. Now if $y$ is $1$st, $3$rd or $5$th vertex of $X_4$, then $(y,6,8,2,x)$ is a path in $G(3,5,n)$. Then $d(x,y)\leq 4$. If $y$ is $2$nd or $4$th vertex of $X_4$, then $(y,12,2,x)$ is a path in $G(3,5,n)$. So $d(x,y)\leq 3$.

{\bf Case VI:}\\
Let $x\in X_1$ and $y\in X_5$. Let the first vertex of $X_5$ be adjacent to vertices of $V_1$ in $G(3,n)$. If $x$ is adjacent to this first vertex, then $d(x,y)\leq 4$. If $|X_5|\geq 3$ and $x$ is adjacent to the third vertex, then $d(x,y)\leq 3$. Suppose $|X_5|\leq 2$ and $x$ is not adjacent to the first vertex, say $z$. Then $(z,24,14,x)$ is a path as $x$ is $2$ or $8$ modulo $10$ and $z$ is adjacent to a vertex of $V_1$ which is $4$ or $6$ modulo $10$. Thus $d(x,y)\leq 4$ as in this case, either $y=z$ or a neighbor of $z$. Now if the first vertex of $X_5$ is adjacent to vertices of $V_2$ in $G(3,n)$, then if $|X_5|\geq 2$ and $x$ is not $4$ or $6$ modulo $10$, then $x$ is adjacent to the $2$nd vertex of $X_5$. Then $d(x,y)\leq 3$. Suppose $|X_5|\geq 2$ and $x$ is $4$ or $6$ modulo $10$. Then as in the last case of Case V, we have $d(x,y)\leq 4$. Finally let $|X_5|=1$ and $x$ is not $4$ or $6$ modulo $10$, then $(y,6,32,26,x)$ is a path in $G(3,5,n)$. So $d(x,y)\leq 4$.

{\bf Case VII:}\\
Let $x\in X_2$ and $y\in X_3$. Since $x$ has neighbors $\set{20,4,16}$ or $\set{32,20,8}$ or $\set{32,8,4,16}$ in\\
$P_0=(32,26,20,14,8,2,4,10,16,22)$ (see Table \ref{tabv1v2}), $d(x,y)\leq 4$ except $y=16,22$ when $x$ is $4$ or $6$ modulo $10$. Since we have a path $(22,16,6,8,x)$ in $G(3,5,n)$, $d(x,y)\leq 4$.

{\bf Case VIII:}\\
Let $x\in X_2$ and $y\in X_4$. This case is similar to Case V except when $x$ is $4$ or $6$ modulo $10$ and it is adjacent to none of the vertices of $X_4$. Also the first vertex of $X_4$ is adjacent to vertices of $V_1$ in $G(3,n)$. Now if $y$ is $1$st, $3$rd or $5$th vertex of $X_4$, then $(y,24,2,8,x)$ is a path in $G(3,5,n)$. If $y$ is $2$nd or $4$th vertex of $X_4$, then $(y,4,2,8,x)$ is a path in $G(3,5,n)$. So $d(x,y)\leq 4$.

{\bf Case IX:}\\
Let $x\in X_2$ and $y\in X_5$. The case is similar to Case VI where the exhibited paths $(z,24,14,x)$ and $(y,6,32,26,x)$ are to be replaced by $(z,6,20,x)$ and $(y,24,2,4,x)$.

{\bf Case X:}\\
Let $x,y\in X_3$. We have $X_3$ is a path $P_0=(32,26,20,14,8,2,4,10,16,22)$ where \\
$\set{26,14,2,10,22}\subset N(12)$. Thus every vertex or its neighbor on the path is adjacent to $12$. Thus $d(x,y)\leq 4$.

{\bf Case XI:}\\
Let $x\in X_3$ and $y\in X_4$. First suppose the first vertex of $X_4$ is adjacent to vertices of $V_1$ in $G(3,n)$. Then $1$st, $3$rd and $5$th vertices of $X_4$ are adjacent to $24$ which is also adjacent to $14,2,10,22$ of $X_3$. Thus these $3$ vertices are at most $4$ distance apart from vertices of $X_3$ except $32$. Now $32$ is adjacent to $30$ which is adjacent to $2$nd and $4$th vertices of $X_4$. Thus $32$ is at most $3$ distance apart from any vertex of $X_4$. Now $30$ is adjacent to $32,8,4,16$ of $X_3$. Thus $2$nd and $4$th vertices are at most $4$ distance apart from any vertex of $X_3$. Next we consider that the first vertex of $X_4$ is adjacent to vertices of $V_2$ in $G(3,n)$. Then $12$ is adjacent to $2$nd and $4$th vertices of $X_4$. Since $12$ is adjacent to $26,14,2,10,22$ in $X_3$, every vertex of $X_4$ is at most $4$ distance apart from any vertex of $X_3$ in this case.

{\bf Case XII:}\\
Let $x\in X_3$ and $y\in X_5$. If $|X_5|\geq 2$, then $d(x,y)\leq 4$ as in Case XI since cases are considered for $1$st, $3$rd and $2$nd vertices of $X_4$ there. Let $|X_5|=1$. If the first and only vertex of $X_5$ is adjacent to vertices of $V_1$ in $G(3,n)$, then $d(x,y)\leq 4$ as above. Let this vertex be adjacent to vertices of $V_2$ in $G(3,n)$. Then this vertex is adjacent to $30$ which is adjacent to $32,8,4,16$ of $X_3$. So any vertex of $X_3$ is at most $4$ distance apart from this one.

{\bf Case XIII:}\\
Let $x,y\in X_4$. If $x$ and $y$ are vertices of the same strip, then they are on a path of length $4$. So $d(x,y)\leq 4$. Suppose they are lying in separate strips, say $X_4$ and $X_4^\prime$.

If the first vertex of $X_4$ and the first vertex of $X_4^\prime$ are adjacent to vertices of $V_1$ in $G(3,n)$, then the $3$rd and $5$th vertices of $X_4$ are also adjacent to the same vertex as in the $1$st one if we choose, say, $24\in V_1$. Similar is the case for $X_4^\prime$. Thus $d(x,y)\leq 4$. Next if the first vertex of $X_4$ and the first vertex of $X_4^\prime$ are adjacent to vertices of $V_2$ in $G(3,n)$, then both of them and $3$rd and $5$th vertices of both strips are adjacent to $6$. So we get $d(x,y)\leq 4$.

Suppose the first vertex of $X_4$ is adjacent to vertices of $V_1$ and the first vertex of $X_4^\prime$ is adjacent to vertices of $V_2$ in $G(3,n)$. Then $30\in V_2$ is adjacent to $2$nd and $4$th vertices of $X_4$ and $1$st and $5$th vertices of $X_4^\prime$. Then $d(x,y)\leq 4$ unless $y$ is the $3$rd vertex of $X_4^\prime$. Now $18\in V_2$ is also adjacent to $2$nd and $4$th vertices of $X_4$ and the $3$rd vertex of $X_4^\prime$. Then $d(x,y)\leq 4$ when $y$ is the $3$rd vertex of $X_4^\prime$.

Finally we consider that the first vertex of $X_4$ is adjacent to vertices of $V_2$ and the first vertex of $X_4^\prime$ is adjacent to vertices of $V_1$ in $G(3,n)$. This case is same as above if we interchange the role of $X_4$ and $X_4^\prime$.

{\bf Case XIV:}\\
Let $x\in X_4$ and $y\in X_5$. This case is similar to Case XIII when first vertices of both the strip are adjacent to $V_1$ (or $V_2$) in $G(3,n)$. Suppose the first vertex of $X_4$ is adjacent to vertices of $V_1$ and the first vertex of $X_5$ is adjacent to vertices of $V_2$ in $G(3,n)$. In this case also $30$ and $18$ resolve the cases as in Case XIII. Next suppose the first vertex of $X_4$ is adjacent to vertices of $V_2$ and the first vertex of $X_5$ is adjacent to vertices of $V_1$ in $G(3,n)$. We again use $30$ and $18$ to get $d(x,y)\leq 4$ unless $|X_5|=1$. So let $|X_5|$=1.

Now first strip on one side of the path $P$ in $G(3,n)$, whose first vertex is adjacent to vertices of $V_2$ in $G(3,n)$, is $(28,34,40,46,52)$ and $(28+60k,34+60k,40+60k,46+60k,52+60k)$, ($k\in\Nat$) thereafter. On the other side of $P$, it starts with $(68,74,80,86,92)$ and $(68+60k,74+60k,80+60k,86+60k,92+60k)$, ($k\in\Nat$) thereafter. It is interesting to note that for each such first set of strips, $2$nd and $4$th vertices are adjacent to vertices in $V_1$ whose values are $10$ less than them respectively. For example, $34\leftrightarrow 24$, $46\leftrightarrow 36$, $94\leftrightarrow 84$, $106\leftrightarrow 96$ etc. Again for each such second set of strips, $2$nd and $4$th vertices are adjacent to vertices in $V_1$ whose values are $10$ more than them respectively. For example, $74\leftrightarrow 84$, $86\leftrightarrow 96$, $134\leftrightarrow 144$, $146\leftrightarrow 156$ etc.

Thus there is a vertex in $V_1$ which is $4$ modulo $10$ (resp. $6$ modulo $10$) but adjacent to the $2$nd vertex (resp. $4$th vertex) of $X_4$ and the only vertex of $X_5$. Thus we have $d(x,y)\leq 3$ in this case.

{\bf Case XV:}\\
Let $x,y\in X_5$. If $x$ and $y$ are in the same strip, then both vertices are on a path of length less than $4$, $d(x,y)< 4$. Suppose they are lying in separate strips, say $X_5$ and $X_5^\prime$. If at least one of $|X_5|$ or $|X_5^\prime|$ is greater than $1$, then we can show $d(x,y)\leq 4$ as in Case XIII or Case XIV. We note that there cannot be two broken strips of one vertex.

As we noted above, strips are started in one side of $P$ with $(28,34,40,46,52)$ and $(28+30k,34+30k,40+30k,46+30k,52+30k)$, ($k\in\Nat$) thereafter, whereas in other side with $(38,44,50,56,62)$ and $(38+30k,44+30k,50+30k,56+30k,62+30k)$, ($k\in\Nat$) thereafter. So at the starting of strips with $38,68,98,\ldots$, the other side last strip contains $2$ or $3$ vertices before the appearance of the $2$nd vertex of this strip. Again at the starting of strips with $58,88,118,\ldots$, the other side last strip contains $4$ or $5$ vertices before the appearance of the $2$nd vertex of this strip. This proves our claim that the case $|X_5|=|X_5^\prime|=1$ does not arise.

We covered all the cases. Thus the graphs $G(3,5,n)$ are connected with diameter at most $4$.
\end{proof}

The above proof shows that graphs $G(3,5,n)$ is far more complicated than graphs $G(3,n)$ or $G(5,n)$. Now we show the Hamiltonian property of graphs $G(3,5,n)$. The complete proof is long and involving. Below we sketch the main idea of the proof.

\begin{theorem}
The graphs $G(3,5,n)$ have Hamiltonian paths for all $n$ and they are Hamiltonian for all even $n>2$.
\end{theorem}

\begin{proof} We first note that $G(3,5,n)=G(3,n)$ for all $n\leq 12$. So the result follows from Theorem \ref{hamg3n} for $n\leq 12$. Also if $G(3,5,2n)$ is Hamiltonian, $G(3,5,2n-1)$ has a Hamiltonian path.

By Theorem \ref{thmg35n}, we see the graph $G(3,5,n)$ has a path of length $10$, several other paths of length $5$ and one or two paths (as the case may be) of length less than $5$. Following Figure \ref{hamg35} we observe a pattern of determining Hamiltonian cycle (arrowhead lines are edges of the Hamiltonian cycle and dotted lines are some other edges in the paths of the graph). There are many other patterns, we choose one which is convenient. An arrangement of vertices such that they form a spanning cycle as we did for $G(3,n)$ in Theorem \ref{hamg3n} is a matter of time and patience for some particular small $n$. But we require a pattern that can be repeated to get larger cycles. The most important and difficult part is that, in $G(3,n)$ all vertices in the path were adjacent to all vertices of independent sets, $V_1$ or $V_2$ (see Theorem \ref{hamg3n}) according to their belonging in partite sets. But here it is not true. Even if a vertex in a path and a vertex in $V_i$, $(i=1$ or $2)$ belong to opposite partite sets, they are not adjacent if sum or difference is a non-trivial multiple of $10$. Thus replacing symbols $\mathbf{U}$ or $\mathbf{V}$ by a number is not easy in this case, which was arbitrary in the case of $G(3,n)$. However we will show an assignment of such numbers is always possible.

Recall that\\ $V_1=\Set{x\in V}{\Mod{x}{0}{12}}$ $=\set{12,24,36,48,\ldots}$, and \\
$V_2=\Set{x\in V}{\Mod{x}{6}{12}}$ $=\set{6,18,30,36,\ldots}$.\\
As before we use $\mathbf{U}$ and $\mathbf{V}$ to represent and an element of $V_1$ and $V_2$ respectively. 

\begin{figure}[ht]
\begin{center}
\includegraphics[scale=0.7]{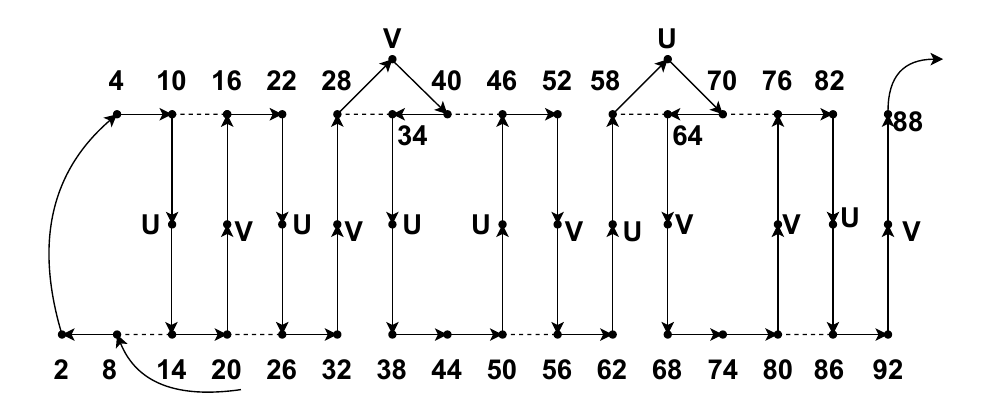}
\caption{The pattern of Hamiltonian cycle in $G(3,5,n)$}\label{hamg35}
\end{center}
\end{figure}

First we consider the following sequence:
{\footnotesize $$\begin{array}{l}
4,10,\mathbf{U}, \\
14,20,\mathbf{V},16,22,\mathbf{U},26,32,\mathbf{V},28,\mathbf{V},40,34,\mathbf{U},38,\\
 44,50,\mathbf{U},46,52,\mathbf{V},56,62,\mathbf{U},58,\mathbf{U},70,64,\mathbf{V},68\\
74,80,\mathbf{V},76,82,\mathbf{U},86,92,\mathbf{V},88,\mathbf{V},100,94,\mathbf{U},98\\
 104,110,\mathbf{U},106,112,\mathbf{V},116,122,\mathbf{U},118,\mathbf{U},130,124,\mathbf{V},128\\
\ldots \\
\mathbf{V},8,2,4. 
\end{array}$$}
One can note that the difference between the corresponding numbers in consecutive rows of the above sequence is $30$. Following this rule we can extend the sequence as large as we want. Since the appearance of members of $V_1$ and $V_2$ will follow (\ref{eqap}), each $\mathbf{U}$ and $\mathbf{V}$ can be replaced by suitable members of $V_1$ and $V_2$. 

We provide Hamiltonian cycles of $G(3,5,n)$ for $14\leq n\leq 44$ in Table \ref{thc35n444} and for $46\leq n\leq 74$ in Table \ref{thc35n4674} by using the above sequence,  the pattern in Figure \ref{hamg35} and its various modifications as required. Then we note that the pattern will repeat for $n=m+30k$ for $46\leq m\leq 74$ and any $k\in\Nat$. We split the cycles into $4$ paths, say, $P_1,P_2,P_3$ and $P_4$ as in Table \ref{tp1234}.

\begin{table}
{\footnotesize $$\begin{array}{|c|c|c|c|c|}
\hline
n & P_1 & P_2 & P_3 & P_4\\
\hline
46 & (4,10,24) & \text{next } 30 \text{ numbers} & \text{next } n-35 \text{ numbers} & (8,2,4)\\
48,72 & (4,10,12) & \text{next } 30 \text{ numbers} & \text{next } n-36 \text{ numbers} & (2,8,18,4)\\ 
50 & (4,10,16,22,12) & \text{next } 30 \text{ numbers} & \text{next } n-40 \text{ numbers} & (6,20,14,8,2,4)\\ 
52,54 & (4,10,12)  & \text{next } 30 \text{ numbers} & \text{next } n-36 \text{ numbers} & (18,8,2,4)\\ 
56 & (4,10,12)  & \text{next } 30 \text{ numbers} & \text{next } n-36 \text{ numbers} & (6,8,2,4)\\ 
58,62,64,66,68,70 & (4,10,12) & \text{next } 30 \text{ numbers} & \text{next } n-35 \text{ numbers} & (8,2,4)\\ 
60,74 & (4,10,12) & \text{next } 30 \text{ numbers} & \text{next } n-36 \text{ numbers} & (2,8,6,4)\\
\hline
\end{array}$$}
\caption{$P_1,P_2,P_3,P_4$ for $46\leq n\leq 74$.}\label{tp1234} 
\end{table}

Now we can easily construct the Hamiltonian cycle of $G(3,5,m+30)$ from the Hamiltonian cycle of $G(3,5,m)$ in the following way. If the Hamiltonian cycle of $G(3,5,m)$ is $(P_1,P_2,P_3,P_4)$, then the same for $G(3,5,m+30)$ is $(P_1,P_2,P_3^\prime,P_4)$, where $P_3^\prime$ is obtained by adding $60$ to the sequence $(P_2,P_3)$.

For example, the Hamiltonian cycle of $G(3,5,48)$ is\\
$C_{48}=(4, 10, 12,\ 14, 20, 78, 16, 22, 36, 26, 32, 30, 28, 6, 40, 34, 24, 38, 44, 50, 72, 46, 52, 42, 56, 62, 60,$\\ 
$58, 48, 70, 64, 54, 68, 74, 80, 66, 76, 82, 96, 86, 92, 90, 88, 94, 84, 2, 8, 18, 4)$\\
with $P_1=(4,10,12)$, $P_2=(14, 20, 78, 16, 22, 36, 26, 32, 30, 28, 18, 40, 34, 24, 38, 44, 50, 72, 46, 52, 42,$\\
$56, 62, 60, 58, 48, 70, 64, 54, 68)$, $P_3=(74, 80, 66, 76, 82, 96, 86, 92, 90, 88, 94, 84)$, and\\
 $P_4=(2, 8, 18, 4)$. Then we compute $P_3^\prime=(14, 20, 78, 16, 22, 36, 26, 32, 30, 28, 18, 40, 34, 24, 38, 44, 50,$\\
$ 72, 46, 52, 42, 56, 62, 60, 58, 48, 70, 64, 54, 68,\ 74, 80, 66, 76, 82, 96, 86, 92, 90, 88, 94, 84) + 60 =$\\
$(74, 80, 138, 76, 82, 96, 86, 92, 90, 88, 78, 100, 94, 84, 98, 104, 110, 132, 106, 112, 102, 116, 122, 120, 118, $\\
$ 108, 130, 124, 114, 128, 134, 140, 126, 136, 142, 156, 146, 152, 150, 148, 154, 144)$.\\
 We claim that the Hamiltonian cycle of $G(3,5,78)$ is $(P_1,P_2,P_3^\prime,P_4)$, i.e.,\\
$(4,10,12,\ 14, 20, 78, 16, 22, 36, 26, 32, 30, 28, 18, 40, 34, 24,38, 44, 50, 72, 46, 52, 42, 56, 62, 60, $\\
$58, 48, 70, 64, 54, 68,\ 74, 80, 138, 76, 82, 96, 86, 92, 90, 88, 78, 100, 94, 84, 98, 104, 110, 132, 106,$\\
$ 112, 102, 116, 122, 120, 118, 108, 130, 124, 114, 128, 134, 140, 126, 136, 142, 156, 146, 152, 150, 148,$\\
$ 154, 144,\ 2, 8, 18, 4)$.

Let us explain the reason for the claim. 

\begin{enumerate}
\item Here $C_{48}$ is a Hamiltonian cycle of $G(3,5,48)$. So it contains all even positive integers from $2$ to $96$ and $P_1,P_2,P_3,P_4$ contain all of them and they appear only once, expect the first and the last number, $4$. We keep $P_1,P_2,P_4$. Only $P_3$ is changed to $(P_2,P_3)+60$.
\item The numbers in $P_1$ or $P_4$ are not appearing in $P_2$ or $P_3^\prime$. Also all numbers $y+60$, where $y$ is either in $P_1$ or $P_4$ are in $P_2$ and we are keeping $P_2$. So these numbers are in the new sequence. Here these numbers are $\set{64,70,72,62,68,78}$. 
\item If the number $x$ is in $P_2$, then $x+60$ is not in $P_2$. So there is no repetition of members of $P_2$ by the construction of $P_3^\prime$. 
\item If the number $z$ is in $P_3$, then $z-60$ is in $P_2$. Thus all these numbers $z$ are in $P_3^\prime$ as we add $60$ to members of $P_2$.
\item The path $P_2$ contains $30$ numbers and we need exactly $30$ new even integers in the Hamiltonian cycle of $G(3,5,78)$. Since numbers of $P_3$ are also added by $60$, they are different from those in $P_2$. As we noted in Column 4 that all numbers of $P_3$ are in $P_3^\prime$. So we get exactly $|P_2-P_3|+|P_3|=|P_2|=30$ new even integers. Since $P_2\cup P_3$ contains distinct even numbers less than or equal to $96$, we have these new numbers are also distinct and less than or equal to $96+60=156=78\times 2$ in $P_3^\prime$. 
\item The most important point is that the assignments for $\mathbf{U}$ or $\mathbf{V}$ by members of $V_1$ or $V_2$ (respectively) remain proper by adding $60$ to elements of $P_2$ and $P_3$ as these numbers were properly assigned for $m=48$ and now these numbers as well as numbers before and after them, all are added by $60$. So they will be in the same congruence classes modulo $10$ as they were before. Moreover, adjacencies between vertices repeat exactly in the same way at the interval of $60$ as the first given sequence shows. The sequence is repeating by $30$ but adding $60$ will keep the numbers in the same partite set as it is a multiple of $4$. So there will be no chance of newly creating non-trivial multiple of $10$ for adding or subtracting them with their before and after numbers. Also if the difference of any two of them was exactly $10$, it is still $10$ after adding. For example, $36$ was between $22$ and $26$ in the cycle of $G(3,5,48)$. Now $96$ is in between $82$ and $86$ in the $P_3^\prime$ part of the cycle of $G(3,5,78)$.
\end{enumerate} 

Thus we have all even numbers between $2$ and $156$ appear in $(P_1,P_2,P_3^\prime,P_4)$ and each number appears only once before the last number $4$ which is the first number of the cycle. Hence this sequence is a Hamiltonian cycle of $G(3,5,78)$. We can continue the process further by the same method.

If we take $m=78$, then $P_1=(4,10,12)$,  $P_2=(14, 20, 78, 16, 22, 36, 26, 32, 30, 28, 18, 40,$\\
$ 34, 24, 38, 44, 50, 72, 46, 52, 42, 56, 62, 60, 58, 48, 70, 64, 54, 68)$, $P_3=(74, 80, 138, 76, 82, 96, 86, 92, 90,$\\
 $88, 78, 100, 94, 84, 98, 104, 110, 132, 106, 112, 102, 116, 122, 120, 118, 108, 130, 124, 114, 128, 134, 140, $\\
$ 126, 136, 142, 156, 146, 152, 150, 148, 154, 144)$ and $P_4=(2, 8, 18, 4)$.\\
 Thus we compute (new) $P_3^\prime=(14, 20, 78, 16, 22, 36, 26, 32, 30, 28, 18, 40, 34, 24, 38, 44, 50, 72, 46, 52,$\\
$ 42, 56, 62, 60, 58, 48, 70, 64, 54, 68,\ 74, 80, 138, 76, 82, 96, 86, 92, 90, 88, 78, 100, 94, 84, 98, 104, 110,$\\
$ 132, 106, 112, 102, 116, 122, 120, 118, 108, 130, 124, 114, 128, 134, 140, 126, 136, 142, 156, 146, 152, 150,$\\
$ 148, 154, 144)+60 = (74, 80, 138, 76, 82, 96, 86, 92, 90, 88, 78, 100, 94, 84, 98, 104, 110, 132, 106, 112,$\\
$ 102, 116, 122, 120, 118, 108, 130, 124, 114, 128, 134, 140, 198, 136, 142, 156, 146, 152, 150, 148, 138, 160,$ \\
$  154, 144,158, 164, 170, 192, 166, 172, 162, 176, 182, 180, 178, 168, 190, 184, 174, 188, 194, 200, 186, 196,$\\
$  202, 216, 206, 212, 210, 208, 214, 204)$.
Then the Hamiltonian cycle of $G(3,5,108)$ is given by:\\
$(4,10,12,\ 14, 20, 78, 16, 22, 36, 26, 32, 30, 28, 18, 40, 34, 24, 38, 44, 50, 72, 46, 52, 42, 56, 62, 60,$\\
$ 58, 48, 70, 64, 54, 68,\ 74, 80, 138, 76, 82, 96, 86, 92, 90, 88, 78, 100, 94, 84, 98, 104, 110, 132, 106, $\\
$112, 102, 116, 122, 120, 118, 108, 130, 124, 114, 128, 134, 140, 198, 136, 142, 156, 146, 152, 150,  $\\
$148, 138, 160, 154, 144,158, 164, 170, 192, 166, 172, 162, 176, 182, 180, 178, 168, 190, 184, 174, 188,  $\\
$194, 200, 186, 196, 202, 216, 206,212, 210, 208, 214, 204,\ 2, 8, 18, 4)$.

One may verify that similar arguments work for all the cases described in Table \ref{tp1234}. Thus this construction provides Hamiltonian cycles of $G(3,5,n)$ for all $n\geq 76$. Since Tables \ref{thc35n444} and \ref{thc35n4674} exhibit Hamiltonian cycles for all $4\leq n\leq 74$, we have $G(3,5,n)$ is Hamiltonian for all even $n>2$.
\end{proof}

\section{General graphs $G(p,n)$, their intersections and conclusion}

In \S \ref{ssstr} to \S \ref{secg5n} we have seen interesting structures and several properties of graphs $G(3,n)$ and $G(5,n)$. In general, graphs $G(p,n)$ also follow similar structure and nice properties. Let $p$ be an odd prime. The graph $G(p,n)=(V,E)$ is a bipartite graph with partite sets $X$ and $Y$ (see Remark \ref{rem1}). The graph has an independent set $V_1\cup V_2=\Set{z\in V}{\Mod{z}{0}{p}}$, where $V_1\subseteq X$ and $V_2\subseteq Y$. Let $\eta(p)$ be the number of unordered additive partitions of $p$ by positive integers into $2$ parts. For example, $\eta(3)=1$ as $3=1+2$, $\eta(5)=2$ as $5=1+4=2+3$ etc. Then the graph has $\eta(p)=\left\lfloor \frac{p}{2}\right\rfloor$ number of (disjoint) paths. Every vertex in each path is adjacent to all members of $V_1$ or $V_2$ as well as all members of other paths according to their belonging in opposite partite sets. Following the proofs of \S \ref{secg3np} and \S \ref{secg5n}, one can show that the graphs $G(p,n)$ is connected with diameter at most $3$. They have Hamiltonian path for any $n$ and they are Hamiltonian for all even $n>2$.

The intersection of graphs $G(p,n)$ for more than one prime are much more complicated as we have seen in \S \ref{secg35n}. We are very much interested to study these intersections as any (finite) near Goldbach graph is the intersection of finite numbers of prime multiple missing graphs. We know that the connectedness of all (finite) near Goldbach graphs is equivalent to the statement that every even integer is sum of two members of $\mathcal{P}$, where $\mathcal{P}$ is the set of odd primes and $1$. In this paper, we proved connectedness of any graph $G(p,n)$ and the intersection graph $G(3,5,n)$. Continuing this study will explore the structures of intersections of graphs with larger number of primes, from which we may find an insight to solve the main problem in future. One surprising observation we made that both Goldbach graphs and near Goldbach graphs are connected with diameter at most $5$ up to $10000$ vertices using machine programming. This indicates a stronger result than connectedness:

\begin{conjecture}
The diameters of (finite) Goldbach graphs and (finite) near Goldbach graphs are bounded by a constant. 
\end{conjecture}

In this paper, we show that diameters of graphs $G(p,n)$ are at most $3$ and that of graphs $G(3,5,n)$ is at most $4$. Further study and big data analysis on intersections of graphs $G(p,n)$ may answer it. In \cite{DGGS}, it was observed that near Goldbach graphs $G(n)$ are Hamiltonian for small even $n>2$ and a Hamiltonian path is given for $G(500)$. Here we proved that graphs $G(3,n)$ and $G(5,n)$ and their intersection are Hamiltonian for even $n>2$ which implies the existence of Hamiltonian paths for odd $n$. The big question is that how far this property is being carried over for larger intersections. Again the big data analysis may help us to know the truth. 
Finally we believe that the structures and patterns of these graphs will unfold many interesting properties and knowledge about various parameters other than those covered in this paper.

\newpage

\section{Appendix}

\begin{table}[h]
$$\begin{array}{|ll|}
\hline
n=4: & (2,8,6,4,2)\\
n=6: & (2,8,6,4,10,12,2)\\
n=8: & (2,8,14,12,10,16,6,4,2)\\
n=10: & (2,8,6,20,14,12,10,16,18,4,2)\\
n=12: & (2,8,6,20,14,12,22,16,18,4,10,24,2)\\
n=14: & (2,8,14,12,26,20,6,28,22,24,10,16,18,4,2)\\
n=16: & (2,8,6,20,14,12,26,32,18,28,22,24,10,16,30,4,2)\\
n=18: & (2,8,6,20,14,12,26,32,18,28,34,24,22,16,30,4,10,36,2)\\
n=20: & (2,8,14,12,26,20,6,32,38,24,34,40,18,28,22,36,10,16,30,4,2)\\
n=22: & (2,8,6,20,14,12,26,32,18,44,38,24,34,40,30,28,22,36,10,16,42,4,2)\\
n=24: & (2,8,6,20,14,12,26,32,18,44,38,24,46,40,30,28,34,36,22,16,42,4,10,48,2)\\
n=26: & (2,8,14,12,26,20,6,32,38,24,50,44,18,52,46,36,34,40,30,28,22,48,10,16,\\
 & \null\hfill 42,4,2)\\
n=28: & (2,8,6,20,14,12,26,32,18,44,38,24,50,56,30,52,46,36,34,40,42,28,22,48,\\
 & \null\hfill 10,16,54,4,2)\\
n=30: & (2,8,6,20,14,12,26,32,18,44,38,24,50,56,30,52,58,36,46,40,42,28,34,48,\\
& \null\hfill 22,16,54,4,10,60,2)\\
n=32: & (2,8,14,12,26,20,6,32,38,24,50,44,18,56,62,36,58,64,30,52,46,48,34,40,\\
& \null\hfill 42,28,22,60,10,16,54,4,2)\\
n=34: & (2,8,6,20,14,12,26,32,18,44,38,24,50,56,30,68,62,36,58,64,42,52,46,48,\\
& \null\hfill 34,40,54,28,22,60,10,16,66,4,2)\\
n=36: & (2,8,6,20,14,12,26,32,18,44,38,24,50,56,30,68,62,36,70,64,42,52,58,48,\\
& \null\hfill 46,40,54,28,34,60,22,16,66,4,10,72,2)\\
n=38: & (2,8,14,12,26,20,6,32,38,24,50,44,18,56,62,36,74,68,30,76,70,48,58,64,\\
& \null\hfill 42,52,46,60,34,40,54,28,22,72,10,16,66,4,2)\\
n=40: & (2,8,6,20,14,12,26,32,18,44,38,24,50,56,30,68,62,36,74,80,42,76,70,48,\\
 & \null\hfill 58,64,54,52,46,60,34,40,66,28,22,72,10,16,78,4,2)\\
\hline
\end{array}$$
\caption{Hamiltonian cycles of $G(3,n)$ for $4\leq n\leq 40$}\label{thamcycleg3n}
\end{table}

\vspace{1em}

\begin{table}
$$\begin{array}{|ll|}
\hline
n=4: & (2,4,6,8,10)\\
n=6: & (2,4,6,8,10,12,2)\\[1em]
n=8: & (2,4,6,8,14,12,10,16,2)\\
n=10: & (2,4,6,8,10,12,14,20,18,16,2)\\
n=12: & (2,4,6,8,10,12,14,24,22,20,18,16,2)\\
n=14: & (2,4,6,8,10,12,14,24,22,20,18,28,26,16,2)\\
n=16: & (2,4,6,8,10,12,14,24,22,32,30,28,26,16,18,20,2)\\[1em]
n=18: & (2,4,6,8,10,12,14,\ 24,22,20,18,16,\ 26,28,34,32,30,36,2)\\
n=20: & (2,4,6,8,10,12,14,\ 24,22,20,18,16,\ 26,28,30,32,34,\ 40,38,36,2)\\
n=22: & (2,4,6,8,10,12,14,\ 24,22,20,18,16,\ 26,28,30,32,34,\ 44,42,40,38,36,2)\\
n=24: & (2,4,6,8,10,12,14,\ 24,22,20,18,16,\ 26,28,30,32,34,\ 44,42,40,38,48,46,\\
& \null\hfill 36,2)\\
n=26: & (2,4,6,8,10,12,14,\ 24,22,20,18,16,\ 26,28,30,32,34,\ 44,42,52,50,48,46,\\
& \null\hfill 36,38,40,2)\\[1em]
n=28: & (2,4,6,8,10,12,14,\ 24,22,20,18,16,\ 26,28,30,32,34,\ 44,42,40,38,36,\ \\
& \null\hfill 46,48,54,52,50,56,2)\\
n=30: & (2,4,6,8,10,12,14,\ 24,22,20,18,16,\ 26,28,30,32,34,\ 44,42,40,38,36,\ \\
& \null\hfill 46,48,50,52,54,\ 60,58,56,2)\\
n=32: & (2,4,6,8,10,12,14,\ 24,22,20,18,16,\ 26,28,30,32,34,\ 44,42,40,38,36,\ \\
& \null\hfill 46,48,50,52,54,\ 64,62,60,58,56,2)\\
n=34: & (2,4,6,8,10,12,14,\ 24,22,20,18,16,\ 26,28,30,32,34,\ 44,42,40,38,36,\ \\
& \null\hfill 46,48,50,52,54,\ 64,62,60,58,68,66,56,2)\\
n=36: & (2,4,6,8,10,12,14,\ 24,22,20,18,16,\ 26,28,30,32,34,\ 44,42,40,38,36,\ \\
&\null\hfill 46,48,50,52,54,\ 64,62,72,70,68,66,56,58,60,2)\\[1em]
n=38: & (2,4,6,8,10,12,14,\ 24,22,20,18,16,\ 26,28,30,32,34,\ 44,42,40,38,36,\ \\
&\null\hfill 46,48,50,52,54,\ 64,62,60,58,56,\ 66,68,74,72,70,76,2)\\
\hline
\end{array}$$
\caption{Hamiltonian cycles of $G(5,n)$ for $4\leq n\leq 38$}\label{thamcycleg5n}
\end{table}

\vspace{1em}

\begin{table}
$$\begin{array}{|ll|}
\hline
n=14: & (4, 10, 16, 6, 28, 18, 8, 14, 20, 26, 12, 22, 24, 2, 4)\\
n=16: & (4, 10, 24, 14, 20, 6, 16, 22, 12, 26, 32, 30, 28, 18, 8, 2, 4)\\
n=18: & (4, 10, 12, 14, 20, 6, 16, 22, 36, 26, 32, 30, 28, 34, 24, 2, 8, 18, 4)\\
n=20: & (4, 10, 16, 22, 12, 26, 36, 38, 24, 34, 40, 18, 28, 30, 32, 6, 20, 14, 8, 2, 4)\\
n=22: & (4, 10, 12, 14, 20, 42, 16, 22, 36, 26, 32, 30, 28, 6, 40, 34, 24, 38, 44, 18, 8, 2, 4)\\
n=24: & (4, 10, 12, 14, 20, 42, 16, 22, 36, 26, 32, 30, 28, 6, 40, 46, 48, 34, 24, 38, 44, 18, 8, 2, 4)\\
n=26: & (4, 10, 12, 14, 20, 42, 16, 22, 36, 26, 32, 30, 28, 18, 40, 34, 24, 38, 44, 50, 48, 46, 52, 6, 8, 2, 4)\\
n=28: & (4, 10, 12, 14, 20, 6, 16, 22, 36, 26, 32, 30, 28, 18, 40, 34, 24, 38, 44, 50, 48, 46, 52, 42, 56, 54,\ \\
&\null\hfill  8, 2, 4)\\
n=30: & (4, 10, 12, 14, 20, 30, 16, 22, 36, 26, 32, 54, 28, 18, 40, 34, 24, \ \\
&\null\hfill 38, 44, 50, 56, 42, 52, 46, 48, 58, 60, 2, 8, 6, 4)\\
n=32: & (4, 10, 12, 14, 20, 6, 16, 22, 36, 26, 32, 30, 28, 18, 40, 34, 24, \ \\
&\null\hfill 38, 44, 50, 48, 46, 52, 42, 56, 62, 60, 58, 64, 54, 8, 2, 4)\\
n=34: & (4, 10, 12, 14, 20, 6, 16, 22, 36, 26, 32, 30, 28, 18, 40, 34, 24, \ \\
&\null\hfill 38, 44, 50, 48, 46, 52, 42, 56, 62, 60, 58, 64, 54, 68, 66, 8, 2, 4)\\
n=36: & (4, 10, 12, 14, 20, 6, 16, 22, 36, 26, 32, 30, 28, 18, 40, 34, 24, 38, 44, 50, 72, 46, 52, 42, 56,\ \\
&\null\hfill  62, 60, 58, 48, 70, 64, 54, 68, 66, 8, 2, 4)\\
n=38: & (4, 10, 12, 14, 20, 6, 16, 22, 36, 26, 32, 30, 28, 18, 40, 34, 24, 38, 44, 50, 72, 46, 52, 42, 56, \\
&\null\hfill  62, 60, 58, 64, 54, 68, 74, 48, 70, 76, 66, 8, 2, 4)\\
n=40: & (4, 10, 12, 14, 20, 6, 16, 22, 36, 26, 32, 30, 28, 18, 40, 34, 24, 38, 44, 50, 72, 46, 52, 42, 56,  \\
&\null\hfill 62, 60, 58, 48, 70, 64, 54, 68, 74, 80, 78, 76, 66, 8, 2, 4)\\
n=42: & (4, 10, 12, 14, 20, 78, 16, 22, 36, 26, 32, 30, 28, 6, 40, 34, 24, 38, 44, 50, 72, 46, 52, 42, 56, \\
&\null\hfill   62, 60, 58, 48, 70, 64, 54, 68, 74, 80, 66, 76, 82, 84, 2, 8, 18, 4)\\
n=44: & (4, 10, 12, 14, 20, 42, 16, 22, 36, 26, 32, 54, 28, 18, 40, 34, 24, 38, 44, 50, 48, 46, 52, 30, 56,  \\
& \null\hfill 62, 60, 58, 64, 70, 76, 78, 88, 66,  68, 74, 80, 86, 72, 82, 84, 2, 8, 6, 4)\\
\hline
\end{array}$$
\caption{Hamiltonian cycles of $G(3,5,n)$ for $14\leq n\leq 44$.}\label{thc35n444}
\end{table}

\begin{table}
{\footnotesize $$\begin{array}{|ll|}
\hline
n=46: & (4, 10, 24, 14, 20, 6, 16, 22, 36, 26, 32, 30, 28, 18, 40, 34, 48, 38, 44, 50, 12, 46, 52, 42, 56, 62, 60, 58, 84, 70, 64,\\
&\null\hfill  54, 68, 74, 80, 78, 76, 82, 72, 86, 92, 90, 88, 66, 8, 2, 4)\\
n=48: & (4, 10, 12, 14, 20, 78, 16, 22, 36, 26, 32, 30, 28, 6, 40, 34, 24, 38, 44, 50, 72, 46, 52, 42, 56, 62, 60, 58, 48, 70, 64, \\
&\null\hfill  54, 68, 74, 80, 66, 76, 82, 96, 86, 92, 90, 88, 94, 84, 2, 8, 18, 4) \\
n=50: & (4, 10, 16, 22, 12, 26, 32, 30, 28, 18, 40, 34, 24, 38, 44, 50, 72, 46, 52, 42, 56, 62, 36, 58, 48, 70, 64, 54, 68, 74, 80, \\
&\null\hfill  66, 76, 82, 60, 86, 96, 98, 84, 94, 100, 78, 88, 90, 92, 6, 20, 14, 8, 2, 4) \\
n=52: & (4, 10, 12, 14, 20, 78, 16, 22, 36, 26, 32, 30, 28, 6, 40, 34, 24, 38, 44, 50, 72, 46, 52, 42, 56, 62, 60, 58, 48, 70, 64, \\
&\null\hfill 54, 68, 74, 80, 102, 76, 82, 96, 86, 92, 90, 88, 66, 100, 94, 84, 98, 104, 18, 8, 2, 4) \\
n=54: & (4, 10, 12, 14, 20, 78, 16, 22, 36, 26, 32, 30, 28, 6, 40, 34, 24, 38, 44, 50, 72, 46, 52, 42, 56, 62, 60, 58, 48, 70, 64, \\
&\null\hfill  54, 68, 74, 80, 102, 76, 82, 96, 86, 92, 90, 88, 66, 100, 106, 108, 94, 84, 98, 104, 18, 8, 2, 4) \\
n=56: & (4, 10, 12, 14, 20, 42, 16, 22, 36, 26, 32, 30, 28, 18, 40, 34, 24, 38, 44, 50, 72, 46, 52, 66, 56, 62, 60, 58, 48, 70, 64,  \\
& \null\hfill 54, 68, 74, 80, 102, 76, 82, 96, 86, 92, 90, 88, 78, 100, 94, 84, 98, 104, 110, 108, 106, 112, 6, 8, 2, 4)\ \\
n=58: & (4, 10, 12, 14, 20, 6, 16, 22, 36, 26, 32, 30, 28, 18, 40, 34, 24, 38, 44, 50, 72, 46, 52, 42, 56, 62, 60, 58, 48, 70, 64, \\
&\null\hfill 54, 68, 74, 80, 66, 76, 82, 96, 86, 92, 114, 88, 78, 100, 94, 84, 98, 104, 110, 108, 106, 112, 102, 116, 90, 8, 2, 4)\\
n=60: & (4, 10, 12, 14, 20, 42, 16, 22, 36, 26, 32, 66, 28, 18, 40, 34, 24, 38, 44, 50, 72, 46, 52, 30, 56, 62, 60, 58, 48, 70, 64, \\
& 54, 68, 74, 80, 102, 76, 82, 96, 86, 92, 114, 88, 78, 100, 94, 84, 98, 104, 110, 116, 90, 112, 106, 108, 118, 120, \\
&\null\hfill  2, 8, 6, 4)\\
n=62: & (4, 10, 12, 14, 20, 6, 16, 22, 36, 26, 32, 30, 28, 18, 40, 34, 24, 38, 44, 50, 72, 46, 52, 42, 56, 62, 60, 58, 48, 70, 64, \\
& 54, 68, 74, 80, 66, 76, 82, 96, 86, 92, 90, 88, 78, 100, 94, 84, 98, 104, 110, 108, 106, 112, 102, 116, 122, 120, 118, \\
&\null\hfill 124, 114, 8, 2, 4)\\
n=64: & (4, 10, 12, 14, 20, 6, 16, 22, 36, 26, 32, 30, 28, 18, 40, 34, 24, 38, 44, 50, 72, 46, 52, 42, 56, 62, 60, 58, 48, 70, 64, \\
&54, 68, 74, 80, 66, 76, 82, 96, 86, 92, 90, 88, 78, 100, 94, 84, 98, 104, 110, 108, 106, 112, 102, 116, 122, 120, 118, \\
& \null\hfill 124, 114, 128, 126, 8, 2, 4)\\
n=66: & (4, 10, 12, 14, 20, 6, 16, 22, 36, 26, 32, 30, 28, 18, 40, 34, 24, 38, 44, 50, 72, 46, 52, 42, 56, 62, 60, 58, 48, 70, 64, \\
& 54, 68, 74, 80, 66, 76, 82, 96, 86, 92, 90, 88, 78, 100, 94, 84, 98, 104, 110, 132, 106, 112, 102, 116, 122, 120, 118, \\
&\null\hfill 108, 130, 124, 114, 128, 126, 8, 2, 4) \\
n=68: & (4, 10, 12, 14, 20, 6, 16, 22, 36, 26, 32, 30, 28, 18, 40, 34, 24, 38, 44, 50, 72, 46, 52, 42, 56, 62, 60, 58, 48, 70, 64, \\
& 54, 68, 74, 80, 66, 76, 82, 96, 86, 92, 90, 88, 78, 100, 94, 84, 98, 104, 110, 132, 106, 112, 102, 116, 122, 120, 118, \\
&\null\hfill 124, 114, 128, 134, 108, 130, 136, 126, 8, 2, 4)\\
n=70: & (4, 10, 12, 14, 20, 6, 16, 22, 36, 26, 32, 30, 28, 18, 40, 34, 24, 38, 44, 50, 72, 46, 52, 42, 56, 62, 60, 58, 48, 70, 64, \\
& 54, 68, 74, 80, 66, 76, 82, 96, 86, 92, 90, 88, 78, 100, 94, 84, 98, 104, 110, 132, 106, 112,102, 116, 122, 120, 118,  \\
&\null\hfill 108, 130, 124, 114, 128, 134, 140, 138, 136, 126, 8, 2, 4)\\
n=72: & (4, 10, 12, 14, 20, 78, 16, 22, 36, 26, 32, 30, 28, 6, 40, 34, 24, 38, 44, 50, 72, 46, 52, 42, 56, 62, 60, 58, 48, 70, 64, \\
& 54, 68, 74, 80, 66, 76, 82, 96, 86, 92, 126, 88, 114, 100, 94, 84, 98, 104, 110, 132, 106, 112, 102, 116, 122, 120, \\
&\null\hfill 118, 108, 130, 124, 90, 128, 134, 140, 138, 136, 142, 144, 2, 8, 18, 4)\\
n=74: & (4, 10, 12, 14, 20, 42, 16, 22, 36, 26, 32, 66, 28, 18, 40, 34, 24, 38, 44, 50, 72, 46, 52, 30, 56, 62, 60, 58, 48, 70, 64, \\
& 54, 68, 74, 80, 102, 76, 82, 96, 86, 92, 114, 88, 78, 100, 94, 84, 98, 104, 110, 108, 106, 112, 90, 116, 122, 120,  \\
&\null\hfill  118, 124, 130, 136, 138, 148, 126, 128, 134, 140, 146, 132, 142, 144, 2, 8, 6, 4)\\
\hline
\end{array}$$}
\caption{Hamiltonian cycles of $G(3,5,n)$ for $46\leq n\leq 74$.}\label{thc35n4674}
\end{table}

\section{Appendix II}

\noindent
In this section we report some recent important development in respect of calculation of degrees of vertices in near Goldbach graphs.

\vspace{1em}
\noindent
We found exact formulas for degrees of any positive even integer $x$ in the near Goldbach graph $G(x/2)$ which can be approximated in the following functions:
$$\eta(x)=3\, \lfloor\frac{x}{12}\rfloor \prod\limits_{\Div{p}{x},\, p>2} \frac{p-1}{p-2}\ \prod\limits_{2<p< \sqrt{x}} \left( 1-\frac{2}{p}\right)\ \approx\ \kappa(x) \frac{x\, e^{-\beta}}{(\log\, x)^2}=\eta_a(x).$$
where $\beta=0.183407$, $\displaystyle{\kappa(x)=\prod\limits_{\Div{p}{x},\, p>2} \frac{p-1}{p-2}}$.

\vspace{1em}
\noindent
The following table shows the comparison between exact values and approximated values.

$$\begin{array}{|c|l|l|l|l|l|l|l|l|l|}
\hline
x & 10^2 & 10^3 & 10^4 & 10^5 & 10^6 & 10^7 & 10^8 & 10^9 & 10^{10} \\
\hline
\text{deg}(x) & 6 & 28 & 127 & 810 & 5402 & 38807 & 291400 & 2274205 & 18200488 \\
\hline
\eta(x) & 4 & 20 & 127 & 820 & 5770 & 42642 & 326294 & 2582599 & 20921398 \\
\hline
\eta_a(x) & 5 & 23 & 130 & 837 & 5815 & 42722 & 327097 & 2584477 & 20934266\\
\hline
\end{array}$$

\noindent
Also we note that the correlation coefficient of $x$ and degree\ $(x)$ is $0.999604$.

\end{document}